\documentclass[a4paper,10pt,normalheadings,makeidx]{article}
\usepackage{graphicx} 
\usepackage{amsmath,amssymb,amsthm} 
\usepackage[utf8]{inputenc}
\usepackage{hyperref}
\usepackage{MnSymbol,bbding,pifont}

\newpage

\newtheorem{Theorem}{Theorem}
\newtheorem{Lemma}{Lemma}
\newtheorem{Question}{Question}
\newtheorem{Claim}{Claim}

\newtheorem{Proposition}{Proposition}

\newtheorem{Conjecture}{Conjecture}
 \title{On orientations maximizing total arc-connectivity}
\author{Florian H\"orsch\\ CISPA Saarbrücken, Germany}

\begin{document}

\maketitle

 \begin{abstract}
For a given digraph $D$ and distinct $u,v \in V(D)$, we denote by $\lambda_D(u,v)$ the local arc-connectivity from $u$ to $v$. Further, we define the total arc-connectivity $tac(D)$ of $D$ to be $\sum_{\{u,v\}\subseteq V(D)}(\lambda_D(u,v)+\lambda_D(v,u))$. We show that, given a graph $G$ and an integer $k$, it is NP-complete to decide whether $G$ has an orientation $\vec{G}$ satisfying $tac(\vec{G})\geq k$. This answers a question of Pekec. On the positive side, we show that the corresponding maximization problem admits a $\frac{2}{3}$-approximation algorithm.
\end{abstract}

\section*{Keywords} 

\begin{itemize}
\item complexity,
\item graph orientation,
\item graph connectivity,
\item approximation.
\end{itemize}
\section{Introduction}
This article is concerned with finding orientations of graphs that satisfy certain connectivity requirements. Any undefined notation can be found in Section \ref{preli}.

Graph orientations and connectivity play an important role in discrete optimization. An important question is which properties of a given undirected graph guarantee the existence of an orientation satisfying a certain prescribed connectivity property. The most fundamental result in this field is the following one due to Robbins \cite{robb}.

\begin{Theorem}\label{rob}
A graph has a strongly connected orientation if and only if it is 2-edge-connected.
\end{Theorem}

The numerous generalizations of Theorem \ref{rob} can roughly be divided into two categories: those that require the orientation to satisfy a vertex-connectivity property and those that require the orientation to satisfy an arc-connectivity property.

For vertex-connectivity, the following result was proven by Thomassen in \cite{CT}, improving on earlier results by Berg and Jord\'an \cite{AlexTibor} and Cheriyan, Durand de Gevigney and Szigeti \cite{cds}.

\begin{Theorem}\label{2vx}
A graph $G$ has a 2-vertex-connected orientation if and only if $G$ is 4-edge-connected and $G-v$ is 2-edge-connected for all $v \in V(G)$.
\end{Theorem}

On the negative side, it was proven by Durand de Gevigney \cite{ODG} that there is no hope to generalize this result to higher vertex-connectivity. He proved the following result.

\begin{Theorem}
For any fixed $k \geq 3$, it is NP-complete to decide whether a given graph has a $k$-vertex-connected orientation.
\end{Theorem}

For arc-connectivity, the situation is much brighter. Indeed, the following natural generalization of Theorem \ref{rob} was proven by Nash-Williams in 1960 \cite{N60}.
\begin{Theorem}
Let $G$ be a graph and $k$ a positive integer. Then $G$ has a $k$-arc-connected orientation if and only if $G$ is $2k$-edge-connected.
\end{Theorem}

Actually, in the same article, Nash-Williams proved the following much stronger result which takes into account local arc-connectivities.

\begin{Theorem}\label{nwfort}
Every graph has a well-balanced orientation.
\end{Theorem}

Since then, attempts to strengthen Theorem \ref{nwfort} have yielded limited success. A collection of slight strengthenings of Theorem \ref{nwfort} was given by Kir\'aly and Szigeti \cite{ks}. On the other hand, it was proven by Bern\'ath et al. \cite{bikks} that it is NP-complete to decide if a given graph has a well-balanced orientation respecting some upper and lower bounds on the outdegrees of every vertex and whether a given graph has a well-balanced orientation minimizing a given weight function on the possible orientations of each edge. Further, Bern\'ath and Joret proved that it is NP-complete to decide if a given graph has a well-balanced orientation respecting a preorientation of some edges \cite{bj}. More recently, the author and Szigeti showed that it is NP-complete to decide if a given graph has a well-balanced orientation respecting only upper bounds on the outdgrees of every vertex \cite{hs}. This work will later play a crucial role in the present article.

In the search of an orientation of a graph whose connectivity properties are roughly speaking optimal, well-balanced orientations certainly play an important role as they have strong connectivity properties, but they also have some weaknesses. This is particularly well visible when considering trees. It is easy to see that every orientation of a tree is well-balanced and hence the fact that an orientation of a tree is well-balanced does not reveal any information about its connectivity properties.

In this article, we deal with another measure of connectivity of a digraph that attempts to overcome this issue. For a digraph $D$, we define its total vertex-connectivity {\it $tvc(D)$} to be $\sum_{\{u,v\}\subseteq V(D)}(\kappa_D(u,v)+\kappa_D(v,u))$ and  its total arc-connectivity {\it $tac(D)$} to be $\sum_{\{u,v\}\subseteq V(D)}(\lambda_D(u,v)+\lambda_D(v,u))$.

Our objective now is, given a graph $G$, to find an orientation $\vec{G}$ of $G$ that maximizes $tvc(\vec{G})$ or $tac(\vec{G})$. Observe that when $G$ is a tree, then $tvc(\vec{G})$ and $tac(\vec{G})$ coincide for any orientation $\vec{G}$ of $G$. Actually, the case of trees has been handled succesfully by Hakimi, Schmeichel and Young \cite{hsy} and independently Henning and Oellermann \cite{ho}. The following is an immediate corollary of a more technical result that will prove useful later on.
\begin{Theorem}
Given a tree $T$, we can find in polynomial time an orientation $\vec{T}$ of $T$ that maximizes $tac(\vec{T})$.
\end{Theorem}

Clearly, for a general digraph $D$, we may have $tvc(D)\neq tac(D)$. In \cite{duraj}, Duraj suggested that it may be NP-complete to find an orientation $\vec{G}$ of a given graph $G$ that maximizes $tvc(\vec{G})$. We think that this can be proven by slightly adapting the reduction of Durand de Gevigney in \cite{ODG}. In \cite{ho}, Henning and Oellermann, studying a closely related parameter, provided some extremal results on $tvc(D)$ for digraphs $D$ with a fixed number of vertices and edges and considered the problem for complete graphs and trees. More recently, Casablanca et al. \cite{cdgmo} provided new bounds for the maximum value of $tvc(\vec{G})$ for the case that $G$ is contained in some restricted class of graphs.

Comparatively little work has been done on maximizing $tac(\vec{G})$ over all orientations of a given graph $G$ and this is the main purpose of the present article. In \cite{ho}, Henning and Oellermann used Theorem \ref{nwfort} to show that every 2-edge-connected graph has an orientation $\vec{G}$ that satisfies $tac(\vec{G})\geq \frac{2}{3}\sum_{\{u,v\}\subseteq V(G)}\lambda_G(u,v)$. According to Bang-Jensen and Gutin \cite{bg}, the question whether  $tac(\vec{G})$ can be maximized efficiently over all orientations $\vec{G}$ of a given graph $G$ was raised by Pekec in 1997. Our first contribution is a negative answer to this question.

\begin{Theorem}\label{mainhard}
Given a graph $G$ and a positive integer $k$, it is NP-complete to decide if $G$ has an orientation $\vec{G}$ that satisfies $tac(\vec{G})\geq k$.
\end{Theorem}

The proof of Theorem \ref{mainhard} is based on the construction in \cite{hs} and pretty involved.

On the positive side, we show that the problem of maximizing $tac(\vec{G})$ can be approximated efficiently.

\begin{Theorem}\label{mainapp}
The problem of maximizing $tac(\vec{G})$ over all orientations $\vec{G}$ of a given graph $G$ admits a $\frac{2}{3}$-approximation algorithm.
\end{Theorem}

The algorithm for Theorem \ref{mainapp} is rather simple and based on combining an algorithmic version of Theorem \ref{nwfort} and the result in \cite{hsy}.

The paper is structured as follows. In Section \ref{preli}, we provide some notation and preliminary results. In Section \ref{red}, we prove Theorem \ref{mainhard}. In Section \ref{appro}, we prove Theorem \ref{mainapp}. Finally, in Section \ref{conc}, we conclude our results and give some open problems.
\section{Preliminaries}\label{preli}
We here collect some preliminaries we need for the main proofs in Sections \ref{red} and \ref{appro}. In Section \ref{not}, we give some basic notation and in Section \ref{prelres}, we collect some preliminary results.
\subsection{Notation}\label{not}
We often use $x$ for a single element set $\{x\}$. When we explicitely define a set $\{x,y\}$, we assume that $x \neq y$. A partition of a set $S$ is a collection of nonempty subsets $(S_1,\ldots,S_t)$ of $S$ such that $\bigcup_{i=1}^tS_i =S$ and $S_i \cap S_j= \emptyset$ for all $i,j \in \{1,\ldots,t\}$ with $i\neq j$.
Let $G$ be a graph. For two disjoint sets $S,S' \subseteq V(G)$, we denote by $\delta_G(S,S')$ the set of edges in $E(G)$ that have one endvertex in $S$ and one endvertex in $S'$ and we use $d_G(S,S')$ for $|\delta_G(S,S')|$. If $d_G(S,S')=0$, we say that $S$ and $S'$ are {\it non-adjacent}. We abbreviate $\delta_G(S,V(G)-S)$ to $\delta_G(S)$ and use $d_G(S)$ for $|\delta_G(S)|$. We further use $N_G(S)$ for the set of vertices $v \in V(G)-S$ that satisfy $d_G(S,v)\geq 1$. We denote by $G[S]$ the subgraph of $G$ {\it induced} on $S$, that is the graph whose vertex set is $S$ and whose edge set contains all edges in $E(G)$ that have both endvertices in $S$. We use $i_G(S)$ for $|E(G[S])|$. For some $u,v \in V(G)$, we say that $S$ is a {\it $u\bar{v}$-set} if $u \in S$ and $v \in V(G)-S$. We let $\lambda_G(u,v)=\min\{d_G(S)| \text{ $S$ is a $u\bar{v}$-set }\}$. For some positive integer $k$, we say that $S$ is {\it $k$-edge-connected in $G$} if $\lambda_G(u,v)\geq k$ holds for all $\{u,v\} \subseteq S$ and we abbreviate 1-edge-connected to {\it connected}. We say that $G$ is {\it k-edge-connected} if $V(G)$ is $k$-edge-connected in $G$. A maximal connected subgraph of $G$ is called a {\it component} of $G$. For two vertices $u,v$, a {\it $uv$-path} is a connected graph $P$ with $d_P(u)=d_P(v)=1$ and $d_P(w)=2$ for all $w \in V(P)-\{u,v\}$. A {\it path} is a $uv$-path for some vertices $u,v$. For a path $P$ and $x,y \in V(P)$. we denote by $P_{x,y}$ the unique minimal connected graph with $\{x,y\}\subseteq V(P_{x,y})\subseteq V(P)$ and $E(P_{x,y})\subseteq E(P)$.

Let $D$ be a digraph. For two disjoint sets $S,S' \subseteq V(G)$, we denote by $\delta_D(S,S')$ the set of arcs in $A(D)$ whose tail is in $S$ and whose head is in $S'$ and we use $d_D(S,S')$ for $|\delta_D(S,S')|$. We abbreviate $\delta_D(S,V(D)-S)$ to $\delta_D^+(S)$ and use $\delta_D^-(S)$ for $\delta_D^+(V(D)-S)$, $d_D^+(S)$ for $|\delta_D^+(S)|$, and $d_D^-(S)$ for $|\delta_D^-(S)|$. We denote by $D[S]$ the subdigraph of $D$ {\it induced} on $S$, that is, the graph whose vertex set is $S$ and whose arc set contains all arcs in $A(D)$ that have both endvertices in $S$. For some $\{u,v\} \subseteq V(D)$, we let $\lambda_D(u,v)=\min\{d_D^+(S)| \text{ $S$ is a $u\bar{v}$-set }\}$. We say that $v$ is {\it reachable} from $u$ in $D$ if $\lambda_D(u,v)\geq 1$. For some positive integer $k$, we say that $S$ is {\it $k$-arc-connected in $D$} if $\min\{\lambda_D(u,v),\lambda_D(v,u)\}\geq k$ holds for all $\{u,v\} \subseteq S$ and we abbreviate 1-arc-connected to {\it strongly connected}. We say that $D$ is {\it k-arc-connected} if $V(D)$ is $k$-arc-connected in $D$. A maximal strongly connected subgraph of $D$ is called a {\it strongly connected component} of $G$. Further, for $u,v \in V(D)$, we denote by $\kappa_D(u,v)$ the maximum integer $\ell$ such that $|V(D)|\geq \ell+1$ and $\lambda_{D-S}(u,v)\geq 1$ for all $S \subseteq V(D)-\{u,v\}$ with $|S|\leq \ell-1$. We say that $D$ is $k$-vertex-connected for some positive integer $k$ if $\min\{\kappa_D(u,v),\kappa_D(v,u)\}\geq k$ holds for every $\{u,v\}\subseteq V(D)$. Recall that $tvc(D)=\sum_{\{u,v\}\subseteq V(D)}(\kappa_D(u,v)+\kappa_D(v,u))$ and   $tac(D)=\sum_{\{u,v\}\subseteq V(D)}(\lambda_D(u,v)+\lambda_D(v,u))$.

An {\it orientation} of a graph $G$ is a digraph $\vec{G}$ that is obtained from $G$ by replacing each edge in $E(G)$ by an arc with the same two endvertices. We say that $\vec{G}$ is {\it well-balanced} if $\min\{\lambda_{\vec{G}}(u,v),\lambda_{\vec{G}}(v,u)\}\geq \lfloor \frac{1}{2} \lambda_G(u,v)\rfloor$ for all $\{u,v\}\subseteq V(G)$. A {\it directed $uv$-path} is an orientation $\vec{P}$ of a $uv$-path $P$ with $\lambda_{\vec{P}}(u,v)=1$.
\subsection{Preliminary results}\label{prelres}
We first need two basic results on digraphs. Proposition \ref{summe} is easy to see and well-known and Proposition \ref{egal} follows from Proposition \ref{summe}.
\begin{Proposition}\label{summe}
Let $D$ be a digraph and $S \subseteq V(D)$. Then $d_D^+(S)-d_D^-(S)=\sum_{v\in S}(d_D^+(v)-d_D^-(v))$.
\end{Proposition}
\begin{Proposition}\label{egal}
Let $D$ be a digraph and $S,X \subseteq V(D)$ such that $d_D^+(x)=d_D^-(x)$ for all $x \in X$. Then $d_D^+(S)-d_D^-(S)=d_D^+(S\cup X)-d_D^-(S\cup X)$.
\end{Proposition}
We next need two results on well-balanced orientations. Both are immediate consequences of the definition of well-balanced orientations.
\begin{Proposition}\label{ext}
Let $G$ be a graph, $a,a' \in V(G)$ with $d_G(a')=d_G(a',a)$ and $\vec{G}$ an orientation of $G$. Then $\vec{G}$ is well-balanced if and only if $\vec{G}-a'$ is a well-balanced orientation of $G-a'$ and $\lfloor\frac{1}{2}d_G(a')\rfloor\leq d_{\vec{G}}(a,a')\leq \lceil\frac{1}{2}d_G(a')\rceil$.
\end{Proposition}
\begin{Proposition}\label{char}
Let $G$ be a graph and $a \in V(G)$ such that $\lambda_G(a,v)=d_G(v)$ holds for all $v \in V(G)-a$. Then an orientation $\vec{G}$ of $G$ is well-balanced if and only if $\min\{\lambda_{\vec{G}}(a,v),\lambda_{\vec{G}}(v,a)\}\geq \lfloor\frac{1}{2}\lambda_G(a,v)\rfloor$ holds for all $v \in V(G)-a$.
\end{Proposition}
The next result on orientations is due to Hakimi \cite{hak}.
\begin{Proposition}\label{hakimi}
Given a graph $G$ and a function $\phi:V(G)\rightarrow \mathbb{Z}_{\geq 0}$, there is an orientation $\vec{G}$ of $G$ with $d_{\vec{G}}^+(v)=\phi(v)$ for all $v \in V(G)$ if and only if $|E(G)|=\sum_{v \in V(G)}\phi(v)$ and $i_G(S)\leq \sum_{v \in S}\phi(v)$ holds for all $S \subseteq V(G)$.
\end{Proposition}

The next result shows that every graph can be decomposed into 2-edge-connected parts. It is both routine and well-known.
\begin{Proposition}\label{2ec}
For every graph $G$, there is a unique partition $(S_1,\ldots,S_t)$ of $V(G)$ such that $G[S_i]$ is 2-edge-connected for $i=1,\ldots,t$ and $\lambda_{G}(u,v)\leq 1$ for every $\{u,v\}\subseteq V(G)$ such that $u \in S_i$ and $v \in S_j$ for some $i,j \in \{1,\ldots,t\}$ with $i \neq j$. Moreover this partition can be computed in polynomial time.
\end{Proposition} 
We further need the following algorithmic result on well-balanced orientations which is due to Gabow \cite{g}.

\begin{Proposition}\label{gabow}
Given a graph $G$, a well-balanced orientation of $G$ can be computed in polynomial time.
\end{Proposition}

Finally, we need the following result of Hakimi, Schleichel and Young from \cite{hsy} mentioned in the introduction.

\begin{Proposition}\label{reach}
Given a graph $G$, we can find in polynomial time an orientation $\vec{G}$ of $G$ that maximizes $\sum_{\{u,v\}\subseteq V(G)}(\min\{\lambda_{\vec{G}}(u,v),1\}+\min\{\lambda_{\vec{G}}(v,u),1\})$ over all orientations of $G$.
\end{Proposition}
\section{The reduction}\label{red}

We here prove Theorem \ref{mainhard}.
Formally, we consider the following problem:

\medskip
Optimally connected orientation ({\bf OCO}):
\medskip

Input: A graph $G$, a positive integer $k$.
\medskip

Question: Is there an orientation $\vec{G}$ of $G$ such that $\sum_{\{u,v\}\subseteq V(G)}(\lambda_{\vec{G}}(u,v)+\lambda_{\vec{G}}(v,u))\geq k$?
\medskip

We show the following restatement of Theorem \ref{mainhard}:

\begin{Theorem}\label{opthard}
OCO is NP-complete.
\end{Theorem}

We consider the following problem:

\medskip
Upper-bounded well-balanced orientation ({\bf UBWBO}):
\medskip

Input: A graph $G$, a function $\ell:V(G)\rightarrow \mathbb{Z}_{\geq 0}$.
\medskip

Question: Is there a well-balanced orientation $\vec{G}$ of $G$ such that $d_{\vec{G}}^+(v)\leq \ell(v)$ for all $v \in V(G)$?
\medskip

The NP-hardness of UBWBO has been proven by the author and Szigeti \cite{hs}.
Actually, we need two technical strengthenings of this result where we restrict the instances to satisfy a collection of technical extra conditions. We call these problems First Special Upper-Bounded Well-balanced Orientation (FSUBWBO) and Second Special Upper-Bounded Well-balanced Orientation (SSUBWBO). The technical definitions of these problems will be postponed to Section \ref{firstsecond2}. While the hardness of FSUBWBO is an immediate consequence of the construction on \cite{hs}, it takes a significant effort to conclude the following hardness result for SSUBWBO from the hardness of FSUBWBO.

\begin{Lemma}\label{2hard}
SSUBWBO is NP-complete.
\end{Lemma}

 After, we conclude the hardness of OCO from the hardness of SSUBWBO.
\smallskip

The rest of this section is structured as follows: In Section  \ref{gadget}, we introduce a gadget that allows us to conclude the hardness of SSUBWBO from the hardness of FSUBWBO which we do in Section \ref{firstsecond2}. Finally, in Section \ref{secondoptimal}, we conclude Theorem \ref{opthard}. 

\subsection{An important gadget}\label{gadget}
In this section, we deal with a gadget that will prove useful in Section \ref{firstsecond2}. We first define this gadget in terms of its important properties and then show its existence.

For two positive integers $\alpha,\beta$, an $(\alpha,\beta)$-gadget is a graph $W$ together with two disjoint, non-adjacent subsets $X=\{x_1,\ldots,x_\alpha\}$ and $Y=\{y_1,\ldots,y_\beta\}$ of $V(W)$ satisfying the following properties:
\begin{itemize}
\item $(1)$: $\lambda_{W}(u,v)=\min\{d_W(u),d_W(v)\}$ for all $\{u,v\}\subseteq V(W)$,
\item $(2)$: $d_W(v)=3$ for all $v \in X \cup Y$,
\item $(3)$: $d_W(v)=4$ for all $v \in V(W)-(X \cup Y)$,
\item $(4)$: for all $S \subseteq V(W)$ which are connected in $W$ with $V(W)-S \neq \emptyset, X \cap S \neq \emptyset$, and $Y \cap S \neq \emptyset$, we have $d_W(S)\geq |Y-S|+2$,
\item $(5)$: for all $S \subseteq V(W)$ with $Y \cap S =\emptyset$, we have $d_W(S)\geq \min\{|X \cap S|,|Y|\}$,
\item $(6)$: for every mapping $\phi:X \cup Y \rightarrow \{1,2\}$ with $|\{v \in X\cup Y|\phi(v)=1\}|=|\{v \in X\cup Y|\phi(v)=2\}|$, there is an orientation $\vec{W}$ of $W$ that satisfies:
\begin{itemize}
\item $(6i)$: $d_{\vec{W}}^+(v)=\phi(v)$ for all $v \in X \cup Y$,
\item $(6ii)$: $d_{\vec{W}}^+(v)=2$ for all $v \in V(G)-(X \cup Y)$,
\item $(6iii)$: $\lambda_{\vec{W}}(u,v)=\min\{d_{\vec{W}}^+(u),d_{\vec{W}}^-(v)\}$ for all $\{u,v\} \subseteq V(W)$.
\end{itemize}
\end{itemize}
\begin{Lemma}\label{gadexists}
For all positive integers $\alpha,\beta$ with $\alpha\equiv \beta (mod \text{ }2)$, there exists an $(\alpha,\beta)$-gadget whose size is polynomial in $\alpha+\beta$.
\end{Lemma}
\begin{proof}
Let $\alpha,\beta$ be positive integers with $\alpha\equiv \beta (mod \text{ }2)$. We now create an $(\alpha,\beta)$-gadget $W$. First, we let $V(W)=\{1,\ldots,\gamma\}\times \{1,\ldots,\gamma\}$ where $\gamma=\alpha+\beta+4$ if $\alpha$ is even and $\gamma=\alpha+\beta+5$ otherwise. We now add an edge linking  two vertices $(j_0,i_0)$ and $(j_1,i_1)$ whenever $|j_0-j_1|+|i_0-i_1|=1$. Next, for $j=1,\ldots,\gamma$, we add an edge linking $(j,1)$ and $(j,\gamma)$. Further, for $i=1,\ldots,\frac{1}{2}(\gamma-\alpha)$, we add an edge linking $(1,\alpha+2i-1)$ and $(1,\alpha+2i)$ and for $i=1,\ldots,\frac{1}{2}(\gamma-\beta)$, we add an edge linking $(\gamma,2i-1)$ and $(\gamma,2i)$. Observe that this creates some pairs of parallel edges. Now for $i=1,\ldots,\alpha$, let $x_i=(1,i)$ and for $i=1,\ldots,\beta$, let $y_i=(\gamma,\gamma-\beta+i)$. Let $X=\{x_1,\ldots,x_\alpha\}$ and $Y=\{y_1,\ldots,y_\beta\}$. This finishes the decription of $(W,X,Y)$. For an illustration, see Figure \ref{cfgz}.

\begin{figure}[h]\begin{center}
  \includegraphics[width=.7\textwidth]{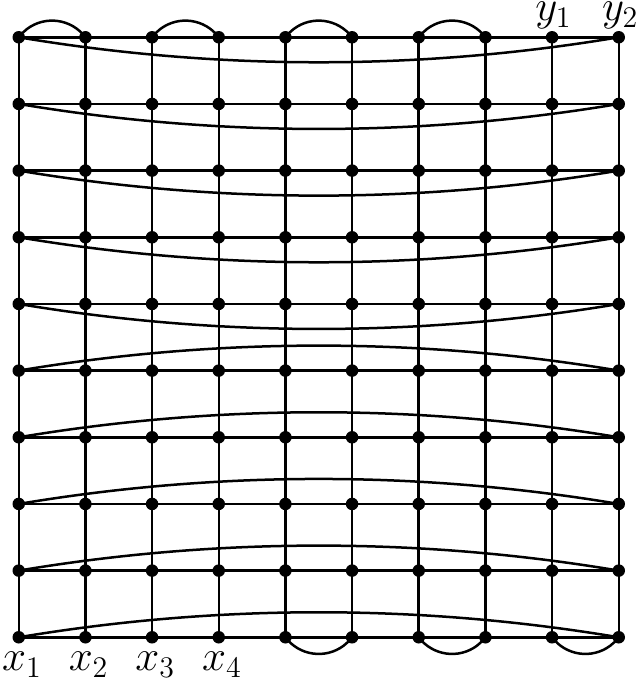} 
  \caption{A $(4,2)$-gadget.}\label{cfgz}
\end{center}
\end{figure}

We show in the following that $(W,X,Y)$ satisfies $(1)-(6)$. We first prove two claims on the sizes of several cuts of $W$ which are key to the proofs of $(1)-(6)$ after. For $i=1,\ldots,\gamma$, let $R_i=\{(i,1),\ldots(i,\gamma)\}$ and $C_i=\{(1,i),\ldots(\gamma,i)\}$. For some $S \subseteq V(W)$, we let $\mathcal{I}_S=\{i \in \{1,\ldots,\gamma\}|C_i \cap S \neq \emptyset\}$ and $\mathcal{J}_S=\{j \in \{1,\ldots,\gamma\}|R_j \cap S \neq \emptyset\}$.

\begin{Claim}\label{row}
Let $S \subseteq V(W)$. Then $d_W(S)\geq \min\{\gamma,|\mathcal{I}_S|+2|\mathcal{J}_S|,|\mathcal{I}_{V(W)-S}|+2|\mathcal{J}_{V(W)-S}|\}$. 
\end{Claim}
\begin{proof}
 If for every $j=1,\ldots,\gamma$, we have that both $R_j \cap S$ and $R_j-S$ are nonempty, as $W[R_j]$ is connected, we obtain that $W[R_j]$ contains an edge of $\delta_W(S)$ for $i=1,\ldots,\gamma$, so $d_W(S)\geq \gamma$. We may hence suppose that there is some $j_0\in \{1,\ldots,\gamma\}$ with $R_{j_0} \subseteq S$ or $R_{j_0} \cap S=\emptyset$. By a similar argument, we may suppose that there is some $i_0 \in \{1,\ldots,\gamma\}$ with $C_{i_0}\subseteq S$ or $C_{i_0} \cap S=\emptyset$. As $R_{j_0} \cap C_{i_0}$ is nonempty and by symmetry, we may suppose that $R_{j_0} \cup C_{i_0} \subseteq V(W)-S$. For every $i \in \mathcal{I}_S$, as $W[C_i]$ is connected, we obtain that $W[C_i]$ contains an edge of $\delta_W(S)$. For every $j \in \mathcal{J}_S$, as $W[R_j]$ is 2-edge-connected, we obtain that $W[R_j]$ contains two edges of $\delta_W(S)$. Hence the statement follows.
\end{proof}
 \begin{Claim}\label{cut2}
Let $\emptyset\neq S \subsetneq V(W)$. Then $d_W(S)\geq \min\{|(X \cup Y)\cap S|,|(X \cup Y)- S|\}+2$. Moreover, if neither $S \subseteq X \cup Y$ and $|(X \cup Y)\cap S|\leq |(X \cup Y)- S|$ nor $V(W)-S \subseteq X \cup Y$ and $|(X \cup Y)- S|\leq |(X \cup Y)\cap S|$ hold, then $d_W(S)\geq \min\{|(X \cup Y)\cap S|,|(X \cup Y)- S|\}+4$.
\end{Claim}
\begin{proof}
By Claim \ref{row} and $\gamma \geq \alpha+\beta+4\geq \min\{|(X \cup Y)\cap S|,|(X \cup Y)- S|\}+4$, we may suppose by symmetry that $d_W(S)\geq |\mathcal{I}_S|+2|\mathcal{J}_S|$. As for every $i \in \{1,\ldots,\gamma\}$, we have $|C_i \cap (X\cup Y)|\leq 1$ and there is at least one $j \in \{1,\ldots,\gamma\}$ with $S \cap R_j \neq \emptyset$, we obtain $d_W(S)\geq |\mathcal{I}_S|+2|\mathcal{J}_S|\geq |(X \cup Y)\cap S|+2\geq \min\{|(X \cup Y)\cap S|,|(X \cup Y)- S|\}+2$. Moreover, if there are distinct $j_1,j_2$ with $R_{j_1}\cap S \neq \emptyset$ and $R_{j_2}\cap S \neq \emptyset$, we obtain $d_W(S)\geq min\{|(X \cup Y)\cap S|,|(X \cup Y)- S|\}+4$. We may hence suppose that there is some $j^*$ with $S \subseteq R_{j^*}$. If $|S-(X\cup Y)|\geq 2$, we obtain $d_W(S)\geq |\mathcal{I}_S|+2|\mathcal{J}_S|=(|S|+2)+2=|S|+4\geq \min\{|(X \cup Y)\cap S|,|(X \cup Y)- S|\}+4$. Next suppose that $S-(X \cup Y)$ contains a single vertex $v$. If $S\cap (X \cup Y)\neq \emptyset$, then $v$ is linked  by at most one edge to $S \cap (X \cup Y)$. As $S-v$ is a real subset of $R_{j^*}$, we obtain $d_W(S-v)\geq |S-v|+2$ by construction.  By construction, this yields 
\begin{align*}
 d_W(S)&=d_W(S-v)+d_W(v)-2d_W(S-v,v)\\
&=(|S-v|+2)+4-2=|S\cap(X\cup Y)|+4\\
&\geq \min\{|(X \cup Y)\cap S|,|(X \cup Y)- S|\}+4.
\end{align*}
Next, if $S=\{v\}$, we have $d_W(S)=4\geq \min\{|(X \cup Y)\cap S|,|(X \cup Y)- S|\}+4$ by construction. 

Finally suppose that $S \subseteq X \cup Y$. If $|(X \cup Y)\cap S|> |(X \cup Y)- S|$, by $|(X \cup Y)\cap S|+|(X \cup Y)- S|\equiv |X \cup Y|\equiv \alpha+\beta\equiv 0\text{ } (mod \text{ }2)$, we obtain $\min\{|(X \cup Y)\cap S|,|(X \cup Y)- S|\}\leq |(X \cup Y)\cap S|-2$. This yields by construction $d_W(S)\geq |\mathcal{I}_S|+2|\mathcal{J}_S|=|(X \cup Y)\cap S|+2\geq \min\{|(X \cup Y)\cap S|,|(X \cup Y)- S|\}+4$. Hence the statement follows.
\end{proof}
We are now ready to prove that $W$ satisfies $(1)-(6)$.
\medskip

$(1):$ Let $u,v \in V(W)$ and let $S \subseteq V(W)$ be a $u\bar{v}$-set. By Claim \ref{row}, by $\gamma\geq 4 \geq \min\{d_W(u),d_W(v)\}$ and by symmetry, we may suppose that $d_W(S)\geq |\mathcal{I}_S|+2|\mathcal{J}_S|$. If $S$ contains at least two elements, we obtain $\max\{|\mathcal{I}_S|,|\mathcal{J}_S|\} \geq 2$ and hence $d_W(S)\geq |\mathcal{I}_S|+2|\mathcal{J}_S|\geq 4 \geq \min\{d_W(u),d_W(v)\}$. Otherwise, we have $d_W(S)=d_W(u)\geq\min\{d_W(u),d_W(v)\}$.
\medskip

$(2)$ and $(3)$ follow immediately by construction.
\medskip

$(4):$ Clearly, we have $\gamma\geq  |Y|+2\geq  |Y-S|+2$. Next, as $S$ is connected, we have $\mathcal{J}_S=\{1,\ldots,\gamma\}$ and hence $|\mathcal{I}_S|+2|\mathcal{J}_S|\geq \gamma \geq |Y-S|+2$. Further, as $|C_i \cap Y|\leq 1$ for every $i \in \{1,\ldots,\gamma\}$ and $V(W)-S \neq \emptyset$, we have $|\mathcal{I}_{V(W)-S}|+2|\mathcal{J}_{V(W)-S}|\geq |\mathcal{I}_{V(W)-S}|+2\geq  |Y-S|+2$. By Claim \ref{row}, we obtain $d_W(S)\geq \min\{\gamma,|\mathcal{I}_S|+2|\mathcal{J}_S|,|\mathcal{I}_{V(W)-S}|+2|\mathcal{J}_{V(W)-S}|\}\geq |Y-S|+2$.
\medskip

$(5):$ By Claim \ref{cut2}, we have $d_W(S)\geq \min\{|(X \cup Y)\cap S|,|(X \cup Y)- S|\}+2\geq \min\{|X\cap S|,|Y- S|\}=\min\{|X\cap S|,|Y|\}.$
\medskip

$(6):$ We define $\phi':V(W)\rightarrow \{1,2\}$ by $\phi'(v)=\phi(v)$ for all $v \in X \cup Y$ and $\phi'(v)=2$ for all $v \in V(W)-(X \cup Y)$.
Now consider some $S \subseteq V(G)$. 
By Claim \ref{cut2}, we have 
\begin{align*}
i_G(S)&=\frac{1}{2}\sum_{v \in S}d_{W[S]}(v)\\
&=\frac{1}{2}(\sum_{v \in S}d_{W}(v)-d_W(S))\\
&\leq \frac{1}{2}(4|S|-|(X \cup Y)\cap S|-\min\{|(X \cup Y)\cap S|,|(X \cup Y)- S|\})\\
&=2|S|-\min\{|(X \cup Y)\cap S|,\frac{1}{2}|X \cup Y|\}\\
&\leq 2|S|-|\{v \in S|\phi'(v)=1\}|\\
&=\sum_{v \in S}\phi'(v).
\end{align*}
Hence by Proposition \ref{hakimi}, we obtain that there is an orientation $\vec{W}$ of $W$ that satisfies $(6i)$ and $(6ii)$. We now show that $\vec{W}$ also satisfies $(6iii)$. Let $\{u,v\} \subseteq V(W)$ and let $S \subseteq V(W)$ be a $u\bar{v}$-set. By symmetry, we may suppose that $|(X \cup Y)\cap S|\leq |(X \cup Y)- S|$. As $d_{\vec{W}}^+(z)=d_{\vec{W}}^-(z)$ for all $z \in V(W)-(X \cup Y)$ and $d_{\vec{W}}^+(z)-d_{\vec{W}}^-(z)\geq -1$ for all $z \in X \cup Y$, by Proposition \ref{summe}, we have 
\begin{align*}
d_{\vec{W}}^+(S)-d_{\vec{W}}^-(S)\geq -|(X \cup Y)\cap S|\tag{$\ostar$}
\end{align*}
 and equality holds if and only if $d_{\vec{W}}^+(z)=1$ for all $z \in (X \cup Y)\cap S$. Further by Claim \ref{cut2} and the assumption, we have

\begin{align*}
d_{\vec{W}}^+(S)+d_{\vec{W}}^-(S)&\geq \min\{|(X \cup Y)\cap S|,|(X \cup Y)- S|\}+2\\
&=|(X \cup Y)\cap S|+2\tag{$\ostar\ostar$},
\end{align*}
and if equality holds, then by the assumption that $|(X \cup Y)\cap S|\leq |(X \cup Y)- S|$, we have $S \subseteq X\cup Y$. 

Summing $(\ostar)$ and $(\ostar \ostar)$, we obtain 

\begin{align*}
d_{\vec{W}}^+(S)\geq 1\tag{$\ostar\ostar\ostar$}.
\end{align*}

If equality holds in $(\ostar\ostar\ostar)$, then equality holds in $(\ostar)$ and $(\ostar \ostar)$, which yields $S \subseteq X \cup Y$ and $d_{\vec{W}}^+(z)=1$ for all $z \in (X \cup Y)\cap S$. In particular, we obtain $d_{\vec{W}}^+(S)\geq 1= d_{\vec{W}}^+(u)\geq\min\{d_{\vec{W}}^+(u),d_{\vec{W}}^-(v)\}$. Otherwise, we have $d_{\vec{W}}^+(S)\geq 2\geq\min\{d_{\vec{W}}^+(u),d_{\vec{W}}^-(v)\}$. This finishes the proof.
\end{proof}

\subsection{Adapted versions of UBWBO}\label{firstsecond2}

We here conclude the hardness of SSUBWBO from FSUBWBO. We now give the technical definitions of these problems.

We denote by FSUBWBO the restriction of UBWBO to instances $(G,\ell)$ for which there is a partition $(V_3,V_3',V_4,a)$ of $V(G)$ satisfying the following properties:
\begin{itemize}
\item $(\alpha)$: $a$ is a single vertex with $d_G(a)> d_G(v)$ and $\lambda_G(a,v)=d_G(v)$ for all $v \in V(G)-a$,
\item $(\beta)$: $d_G(v)=3$ for all $v \in V_3 \cup V_3'$,
\item $(\gamma)$: for all $v \in V_4$, we have $d_G(v)\geq 4$ and $d_G(v)$ is even,
\item $(\delta)$: for all $v \in V_3$, we have $N_G(v)\subseteq V_4\cup a$,
\item $(\epsilon)$: for all $v \in V_3'$, we have $N_G(v)\cap (V_4 \cup a)\neq \emptyset$,
\item $(\zeta)$: $\ell(a)=k$ for some $k \geq \frac{1}{2}d_G(a)$, $\ell(v)=1$ for all $v \in V_3$ and $\ell(v)=d_G(v)$ for all $v \in V_3'\cup V_4$,
\item $(\eta)$: if $(G,\ell)$ is a positive instance of UBWBO, then there is a well-balanced orientation $\vec{G}$ of $G$ with $d_{\vec{G}}^+(v)=1$ for all $v \in V_3$ and $d_{\vec{G}}^+(a)=k$.
\end{itemize}
The following hardness result immediately follows from the construction in the hardness proof of UBWBO in \cite{hs}.

\begin{Proposition}
FSUBWBO is NP-complete.
\end{Proposition}
We now consider the following even more restricted version of UBWBO.
We denote by SSUBWBO the restriction of UBWBO to instances $(G,\ell)$ for which there is a partition $(V_3,V_3',V_4,\{a,a'\})$ of $V(G)$ satisfying the following properties:
\begin{itemize}
\item $(a)$: $a$ is a single vertex with $d_G(a)> d_G(v)$ and $\lambda_G(a,v)=d_G(v)$ for all $v \in V(G)-a$,
\item $(b)$: $a'$ is a single vertex with $d_G(a')=3=d_G(a,a')$,
\item $(c)$: $d_G(v)=3$ for all $v \in V_3 \cup V_3'$,
\item $(d)$: for all $v \in V_4$, we have $d_G(v)\geq 4$, $d_G(v)$ is even and $|\{u \in V(G)|d_G(u)=d_G(v)\}|\geq 3$,
\item $(e)$: for all $v \in V_3$, we have $N_G(v)\subseteq V_4$,
\item $(f)$: for all $v \in V_3'$, we have $N_G(v)\cap V_4\neq \emptyset$,
\item $(g)$: $\ell(v)=1$ for all $v \in V_3$ and $\ell(v)=d_G(v)$ for all $v \in V_3'\cup V_4 \cup \{a,a'\}$.
\end{itemize}

The rest of Section \ref{firstsecond} is concerned with proving  Lemma \ref{2hard} using a reduction from FSUBWBO.
\medskip

 Let $(G_0,\ell_0)$ be an instance of FSUBWBO and let $(V_3,V_3',V_4,a)$ be a partition of $V(G_0)$ as described in the definition of FSUBWBO. We will transform $(G_0,\ell_0)$ into an instance of SSUBWBO in two steps. We first create an instance $(G_1,\ell_1)$ of UBWBO which is equivalent to $(G_0,\ell_0)$ and has the main additional feature that the degree constraint only takes into account vertices of degree 3. Proving the properties of $(G_1,\ell_1)$ is the main technical part of Section \ref{firstsecond2}. After, it is not difficult to transform $(G_1,\ell_1)$ into an equivalent instance $(G_2,\ell_2)$ of SSUBWBO.

For convenience, we let $\mu=\max\{d_{G_0}(v): v \in V_4\}$ and we let $b \in V_4$ be a vertex with $d_{G_0}(b)=\mu$. We now create $G_1$ from $G_0$ in the following way: We first delete all the edges in $\delta_{G_0}(a)$ and add a disjoint $(2k+\mu-d_{G_0}(a),d_{G_0}(a))$-gadget $(W,X,Y)$ whose size is polynomial in the size of $G_0$. Observe that such a gadget exists as $k \geq \frac{1}{2}d_{G_0}(a), 2k+\mu-d_{G_0}(a)+d_{G_0}(a)\equiv \mu \equiv d_{G_0}(b)\equiv 0 \text{ }(mod \text{ } 2)$ by $(\gamma)$ and by Lemma \ref{gadexists}. We now add an edge linking $a$ and $x$ for all $x \in X$ and for every edge $av \in \delta_{G_0}(a)$, we add an edge $yv$ for some $y \in Y$. We do this in a way that we add exactly one new incident edge to every $y \in Y$ which is possible as $|Y|=d_{G_0}(a)$. This finishes the description of $G_1$. We further define $\ell_1:V(G_1)\rightarrow \mathbb{Z}_{\geq 0}$ by $\ell_1(v)=1$ for all $v \in V_3$ and $\ell_1(v)=d_G(v)$ for all $v \in V_3'\cup V_4 \cup \{a,a'\}$. We now give a result that will be useful for proving several of the important properties of $G_1$.

\begin{Claim}\label{cutvergleich}
Let $S \subseteq V(G_1)$ be a set that is connected in $G_1$ with $a \in S$. Then either $d_{G_1}(S)\geq d_{G_0}(S-V(W))$ or $S \subseteq (V(W)-Y)\cup a$ and $d_{G_1}(S)\geq \min\{d_{G_0}(a),d_{G_1}(a)\}$.
\end{Claim}
\begin{proof}
First suppose that $S \subseteq (V(W)-Y)\cup a$. By $(5)$, we obtain $d_W(S\cap V(W))\geq \min \{|X \cap S|,|Y|\}$. If $d_W(S\cap V(W))\geq |X \cap S|$, we obtain by construction that $d_{G_1}(S)\geq d_{G_1}(a,V(W)-S)+d_W(S\cap V(W))\geq |X - S|+|X \cap S|=|X|=d_{G_1}(a)\geq \min\{d_{G_0}(a),d_{G_1}(a)\}$. Otherwise, we have $d_{G_1}(S)\geq d_W(S\cap V(W))\geq |Y|=d_{G_0}(a)\geq \min\{d_{G_0}(a),d_{G_1}(a)\}$. 

Now suppose that $S -((V(W)-Y)\cup a)\neq \emptyset$. As $S$ is connected in $G_1$  and by construction, we obtain that there is some $S'\subseteq V(W)\cap S$ with $X \cap S' \neq \emptyset$ and $Y \cap S' \neq \emptyset$ such that $W[S']$ is a component of $W[V(W)\cap S]$. By $(4)$, we obtain $d_W(V(W)\cap S)\geq d_W(S')\geq |Y-S'| \geq |Y-S|$. Let $Y_1$ be the set of vertices in $Y \cap S$ whose unique neighbor in $V(G_1)-(V(W)\cup a)$ is not contained in $S$. By construction, we obtain 
\begin{align*}
d_{G_1}(S)&\geq d_W(V(W)\cap S)+d_{G_1}(Y_1,V(G_1)-(V(W)\cup S))+d_{G_1-(V(W)\cup a)}(S-(V(W)\cup a))\\
&\geq |Y-S|+|Y_1|+d_{G_0-a}(S-(V(W)\cup a))\\
&\geq d_{G_0}(a,V(G_0)-S)+d_{G_0-a}(S-(V(W)\cup a))\\
&=d_{G_0}(S-V(W)).
\end{align*}
\end{proof}

The next result is a multi purpose one. First, it corresponds to one of the desired properties of $G_2$ and hence is important in its own respect and secondly, it will be used in the forthcoming proofs.

\begin{Claim}\label{cong1}
For all $v \in V(G_1)-a$, we have $\lambda_{G_1}(a,v)=d_{G_1}(v)$.
\end{Claim}
\begin{proof}
First consider some $v \in V(G_0)-a$ and let $S\subseteq V(G_1)$ be an $a\bar{v}$-set. Clearly, we may suppose that $S$ is connected in $G_1$. By Claim \ref{cutvergleich}, construction, $(\alpha),k \geq \frac{1}{2}d_{G_0}(a)$ and the definition of $\mu$, we obtain 
\begin{align*}
d_{G_1}(S)&\geq \min\{d_{G_0}(S-V(W)),d_{G_0}(a),d_{G_1}(a)\}\\
&\geq \min\{\lambda_{G_0}(a,v),d_{G_0}(a),2k+\mu-d_{G_0}(a)\}\\
&\geq \min\{d_{G_0}(v),d_{G_0}(a),\mu\}\\
&=d_{G_0}(v)\\
&=d_{G_1}(v).
\end{align*}

Now consider some $v \in V(W)$ and let $S \subseteq V(G_1)$ be an $a\bar{v}$-set. We may suppose that $S$ is connected in $G_1$. If $V_4-S$ contains a vertex $v_1$, by the above, $(\gamma),(2)$, and $(3)$, we obtain  $d_{G_1}(S)\geq \lambda_{G_1}(a,v_1)\geq d_{G_1}(v_1)\geq 4 \geq d_{G_1}(v)$. We may hence suppose that $V_4 \subseteq S$. As $S$ is connected in $G_1$, as $X$ and $Y$ are disjoint and non-adjacent, and by construction, we obtain that there is some $z \in S\cap (V(W)-(X \cup Y))$. If there is some  $v_2 \in V(W)-(X \cup Y \cup S)$, we obtain by $(1),(3)$ and construction that $d_{G_1}(S)\geq \lambda_{G_1}(v_2,z)\geq \lambda_{W}(v_2,z)\geq \min\{d_W(v_2),d_W(z)\}=4=d_{G_1}(v)$.

 We may hence suppose that $V(W)-(X \cup Y \cup S)=\emptyset$. Next, if there is some $v_3 \in X-S$, we obtain by construction, $(1),(2)$, and $(3)$ that $d_{G_1}(S)\geq d_W(V(W)\cap S)+|\{av_3\}|\geq 3+1=4= d_{G_1}(v)$. We hence may suppose that $V(W)-(Y \cup S)=\emptyset$.

Next suppose that  $|Y-S|\geq 2$. If $|Y \cap S|=\emptyset$, by $(5),X \subseteq S,(\alpha),k \geq \frac{1}{2}d_{G_0}(a),(2)$, and $(3)$, we obtain that $d_{G_1}(S)\geq d_W(V(W)\cap S)\geq \min\{|X \cap S|, |Y|\}=\min\{|X|,|Y|\}\geq \min\{d_{G_0}(a),d_{G_1}(a)\}\geq \min\{d_{G_0}(a),\mu\} \geq 4\geq d_{G_1}(v)$. Otherwise, it follows by $(4),(2)$, and $(3)$ that $d_{G_1}(S)\geq d_W(V(W)\cap S)\geq |Y-S|+2\geq 4= d_{G_1}(v)$.

We may hence suppose that $V(W)-S=\{v\}$. Let $v_4$ be the unique neighbor of $v$ in $G_1$ which is not contained in $V(W)$. If $v_4 \in S$, we obtain by construction, $(2)$, and $(3)$ that $d_{G_1}(S)\geq d_W(V(W)\cap S)+|\{vv_4\}|=3+1=4=d_{G_1}(v)$.

We may hence suppose that $v_4 \in V(G_1)-S$. By $(\alpha),(\beta),(\gamma)$ and the above, we have $d_{G_1}(V(G_1)-(S\cup v))\geq \lambda_{G_1}(a,v_4)= d_{G_1}(v_4)\geq 3$. As $V(W)-S=\{v\}$, by $(2)$ and $(3)$, this yields $d_{G_1}(S)=d_{G_1}(V(G_1)-S)=d_{G_1}(V(G_1)-(S\cup v))+d_{G_1}(v)-2d_{G_1}(V(G_1)-(S\cup v),v)\geq 3+4-2=5\geq d_{G_1}(v)$.
\end{proof}

We are now ready to prove our main result on $(G_1,\ell_1)$.

\begin{Lemma}\label{01}
$(G_1,\ell_1)$ is a positive instance of UBWBO if and only if $(G_0,\ell_0)$ is a positive instance of UBWBO.
\end{Lemma}
\begin{proof}
First suppose that $(G_1,\ell_1)$ is a positive instance of UBWBO, so by $(\alpha)$, construction, Claim \ref{cong1} and Proposition \ref{char}, there is an orientation $\vec{G}_1$ of $G_1$ such that $\min\{\lambda_{\vec{G}_1}(a,v),\lambda_{\vec{G}_1}(v,a)\}\geq \lfloor\frac{1}{2}d_{G_1}(v)\rfloor$ for all $v \in V(G_1)-a$ and $d_{\vec{G_1}}^+(v)\leq \ell_1(v)$ for all $v \in V(G_1)$. Let $\vec{G}_0$ be obtained from $\vec{G}_1$ by contracting $V(W)\cup a$ into $a$ and observe that $\vec{G}_0$ is an orientation of $G_0$. We show in the following that $G_0$ has the desired properties.

We first show that $\vec{G}_0$ is well-balanced. Let $v \in V(G_0)-a$ and let $S \subseteq V(G_0)$ be an $a\bar{v}$-set. By construction, we then have
 \begin{align*}d_{\vec{G_0}}^+(S)&=d_{\vec{G_1}}^+(S \cup V(W))\\ & \geq \lambda_{\vec{G}_1}(a,v)\\ 
&\geq\lfloor\frac{1}{2}d_{G_1}(v)\rfloor\\ &=\lfloor\frac{1}{2}d_{G_0}(v)\rfloor.
\end{align*}
This yields $\lambda_{\vec{G}_0}(a,v)\geq  \lfloor\frac{1}{2}d_{G_0}(v)\rfloor$. We similarly obtain $\lambda_{\vec{G}_0}(v,a)\geq  \lfloor\frac{1}{2}d_{G_0}(v)\rfloor$. 
 Hence $\vec{G}_0$ is well-balanced by $(\alpha)$, construction, Claim \ref{cong1} and Proposition \ref{char}.
\medskip

We still need to show that $d_{\vec{G}_0}^+(v)\leq \ell_0(v)$ for all $v \in V(G_0)$. For all $v \in V(G_0)-(V_3 \cup a)$, we trivially have $d_{\vec{G}_0}^+(v)\leq d_{G_0}(v)= \ell_0(v)$. For all $v \in V_3$, we have $d_{\vec{G}_0}^+(v)\leq d_{\vec{G}_1}^+(v)\leq \ell_1(v) =\ell_0(v)$ by construction and as $ d_{\vec{G}_1}^+(v)\leq \ell_1(v)$. 
Next, we clearly have 

\begin{equation}
d_{\vec{G}_0}^+(a)+d_{\vec{G}_0}^-(a)=d_{G_0}(a).  \tag{$\star$}
\end{equation}

Next, by the definition of $b$ and $\mu$, and $(\gamma),$ we have $d_{\vec{G}_1}^-(a)\geq \lambda_{\vec{G}_1}(b,a)\geq \lfloor\frac{1}{2}d_{G_1}(b)\rfloor =\frac{1}{2}\mu$. We obtain $d_{\vec{G}_1}^+(a)-d_{\vec{G}_1}^-(a)=d_{G_1}(a)-2d_{\vec{G}_1}^-(a)\leq d_{G_1}(a)-\mu=2k-d_{G_0}(a)$. Further, as $\vec{G_1}$ is well-balanced, by $(2),(3)$ and Claim \ref{cong1}, for every $v \in V(W)$, we have
\begin{align*}
 d_{\vec{G_1}}^+(v)+d_{\vec{G_1}}^-(v)&\geq \lambda_{\vec{G_1}}(v,a)+\lambda_{\vec{G_1}}(a,v)\\
& \geq 2\lfloor\frac{1}{2}d_{G_1}(v)\rfloor\\
&=d_{G_1}(v)\\
&=d_{\vec{G_1}}^+(v)+d_{\vec{G_1}}^-(v).
\end{align*}
Hence equality holds throughout and we obtain $d_{\vec{G_1}}^+(v)=d_{\vec{G_1}}^-(v)$.

By Proposition \ref{egal} and construction, we obtain 

\begin{align*}
d_{\vec{G}_0}^+(a)-d_{\vec{G}_0}^-(a)&=d_{\vec{G}_1}^+(V(W) \cup a)-d_{\vec{G}_1}^-(V(W) \cup a)\\  \tag{$\star \star$}
&=d_{\vec{G}_1}^+(a)-d_{\vec{G}_1}^-(a)\\
&\leq 2k-d_{G_0}(a).
\end{align*}

Summing $(\star)$ and $(\star \star)$, we obtain $d_{\vec{G_0}}^+(a)\leq k$. Hence $(G_0,\ell_0)$ is a positive instance of UBWBO.
\bigskip

Now suppose that $(G_0,\ell_0)$ is a positive instance of FSUBWBO, so by $(\alpha)$, construction, and Proposition \ref{char}, there is an orientation $\vec{G}_0$ of $G_0$ such that $\min\{\lambda_{\vec{G}_0}(a,v),\lambda_{\vec{G}_0}(v,a)\}\geq \lfloor\frac{1}{2}d_{G_0}(v)\rfloor$ for all $v \in V(G_0)-a$ and $d_{\vec{G_0}}^+(v)\leq \ell_0(v)$ for all $v \in V(G_0)$. By $(\eta)$, we may suppose that $d_{\vec{G_0}}^+(a)=k$. We let $Y'$ contain every $y \in Y$ whose unique neighbor in $V(G_1)-V(W)$ is a vertex $v$ such that the edge $av$ is oriented as $av$ in $\vec{G_0}$ and we set $Y''=Y-Y'$. We next choose an arbitrary partition $(X',X'')$ of $X$ such that $|X'|=\frac{1}{2}\mu$ and $|X''|=2k-d_{G_0}(a)+\frac{1}{2}\mu$. Observe that such a partition exists as $\mu$ is even by $(\gamma)$. We now define a function $\phi:X \cup Y\rightarrow \{1,2\}$. We set $\phi(v)=1$ for all $v \in X'\cup Y'$ and $\phi(v)=2$ for all $v \in X'' \cup Y''.$

Observe that, by  $d_{\vec{G_0}}^+(a)=k$ and $d_{G_0}(a)=d_{\vec{G_0}}^+(a)+d_{\vec{G_0}}^-(a)$, we have 

\begin{align*}
|\{v \in X \cup Y|\phi(v)=1\}|&=|X'\cup Y'|\\
&=\frac{1}{2}\mu+k\\
&=(2k-d_{G_0}(a)+\frac{1}{2}\mu)+d_{\vec{G_0}}^-(a)\\
&=|X''\cup Y''|\\
&=|\{v \in X \cup Y|\phi(v)=2\}|.
\end{align*}
We hence obtain by $(6)$ that there is an orientation $\vec{W}$ of $W$ that satisfies $(6i),(6ii)$, and $(6iii)$ with respect to $\phi$. 

We are now ready to define an orientation $\vec{G_1}$ of $G_1$. First we orient all edges linking $V(G_1)-V(W)$ and $X'\cup Y'$ away from $X'\cup Y'$ and we orient all edges linking $V(G_1)-V(W)$ and $X''\cup Y''$ toward $X'' \cup Y''$. Next, we give all edges in $E(G_1[V(W)])$ the orientation they have in $\vec{W}$. 
 Finally, we give all the edges both of whose endvertices are contained in $V(G_1)-(V(W)\cup a)$ the orientation they have in $\vec{G}_0$. This finishes the description of $\vec{G_1}$. We show in the following that $\vec{G_1}$ has the desired properties.

For every $v \in V(G_1)-V_3$, we trivially have $d_{\vec{G_1}}^+(v)\leq d_{G_1}(v)=\ell_1(v)$. For all $v \in V_3$, by construction, we have $d_{\vec{G_1}}^+(v)=d_{\vec{G_0}}^+(v)\leq \ell_0(v)=\ell_1(v)$. It hence remains to show that $\vec{G_1}$ is well-balanced. 

\begin{Claim}\label{gross}
Let $v \in V(G)-(V(W)\cup a)$ and let $S\subseteq V(G_1)$ be an $a\bar{v}$-set that is connected in $G$. Then $\min\{d_{\vec{G_1}}^+(S),d_{\vec{G_1}}^-(S)\}\geq \lfloor\frac{1}{2}d_{G_1}(v)\rfloor$.
\end{Claim}
\begin{proof}
Observe that $d_{\vec{G_1}}^+(w)=d_{\vec{G_1}}^-(w)$ for all $w \in V(W)$ by $(6i),(6ii)$ and construction. Hence, by Proposition \ref{egal} and by construction, we have 
\begin{align*}
d_{\vec{G_1}}^+(S)-d_{\vec{G_1}}^-(S)&=d_{\vec{G_1}}^+(S\cup V(W))-d_{\vec{G_1}}^-(S\cup V(W))\\
&=d_{\vec{G_0}}^+(S-V(W))-d_{\vec{G_0}}^-(S-V(W)). \tag{$\ast$}
\end{align*}
 By Claim \ref{cutvergleich}, we obtain that  either $d_{G_1}(S)\geq d_{G_0}(S-V(W))$ or $S \subseteq (V(W)-Y)\cup a$ and $d_{G_1}(S)\geq \min\{d_{G_0}(a),d_{G_1}(a)\}$.
\medskip

First suppose that $d_{G_1}(S)\geq d_{G_0}(S-V(W))$. We then have
\begin{align*}
d_{\vec{G_1}}^+(S)+d_{\vec{G_1}}^-(S)\geq d_{\vec{G_0}}^+(S-V(W))+d_{\vec{G_0}}^-(S-V(W)) \tag{$\ast \ast$}
\end{align*}
Summing $(\ast)$ and $(\ast \ast)$, we obtain $d_{\vec{G_1}}^+(S)\geq d_{\vec{G_0}}^+(S-V(W))$. By construction, this yields
\begin{align*}
d_{\vec{G_1}}^+(S)&\geq d_{\vec{G_0}}^+(S-V(W))\\
&\geq \lambda_{\vec{G_0}}(a,v)\\
&\geq \lfloor\frac{1}{2}d_{G_0}(v)\rfloor\\
&=\lfloor\frac{1}{2}d_{G_1}(v)\rfloor.
\end{align*}
A similar argument shows that $d_{\vec{G_1}}^-(S)\geq \lfloor\frac{1}{2}d_{G_1}(v)\rfloor$.
\medskip

Now suppose that $S \subseteq (V(W)-Y)\cup a$ and $d_{G_1}(S)\geq \min\{d_{G_0}(a),d_{G_1}(a)\}$. By $(\ast)$ and $d_{\vec{G_0}}^+(a)=k$, we obtain 
\begin{align*}
d_{\vec{G_1}}^+(S)-d_{\vec{G_1}}^-(S)&=d_{\vec{G_0}}^+(a)-d_{\vec{G_0}}^-(a)\\
&=2k-d_{G_0}(a). \tag{$\pentagram$}
\end{align*}
Further, we have 

\begin{align*}
d_{\vec{G_1}}^+(S)+d_{\vec{G_1}}^-(S)&\geq \min\{d_{G_0}(a),d_{G_1}(a)\}\\
&=\min\{d_{G_0}(a),2k+\mu-d_{G_0}(a)\}.
 \tag{$\pentagram \pentagram$}
\end{align*}

Summing $(\pentagram)$ and $(\pentagram \pentagram)$, by $k \geq \frac{1}{2}d_{G_0}(a)$ and the definition of $\mu$, we obtain

\begin{align*}
 d_{\vec{G_1}}^+(S)&\geq \frac{1}{2}(2k-d_{G_0}(a)+\min\{d_{G_0}(a),2k+\mu-d_{G_0}(a)\})\\
&\geq \frac{1}{2} \min\{d_{G_0}(a),2k+\mu-d_{G_0}(a)\}\\
&\geq \frac{1}{2}\mu\\
&\geq \frac{1}{2}d_{G_0}(v)\\
&=\frac{1}{2}d_{G_1}(v).
\end{align*}
Summing $-(\pentagram)$ and $(\pentagram \pentagram)$, by $d_{\vec{G_0}}^+(a)=k$, and the definition of $\mu$, we obtain
\begin{align*}
d_{\vec{G_1}}^-(S)&\geq \frac{1}{2}(d_{G_0}(a)-2k+\min\{d_{G_0}(a),2k+\mu-d_{G_0}(a)\})\\
&=\min\{d_{G_0}(a)-k,\frac{1}{2}\mu\}\\
&= \min\{d_{\vec{G_0}}^-(a),\frac{1}{2}\mu\}\\
&\geq \min\{\lambda_{\vec{G_0}}(v,a),\frac{1}{2}\mu\}\\
&\geq \lfloor\frac{1}{2}d_{G_0}(v)\rfloor\\
&=\lfloor\frac{1}{2}d_{G_1}(v)\rfloor.
\end{align*}
\end{proof}

\begin{Claim}\label{klein}
Let $v \in V(W)$ and let $S\subseteq V(G_1)$ be an $a\bar{v}$-set that is connected in $G$. Then $\min\{d_{\vec{G_1}}^+(S),d_{\vec{G_1}}^-(S)\}\geq 2$.
\end{Claim}
\begin{proof}

We first show that $d_{\vec{G_1}}^+(S)\geq 2$.
 If $X'' \cap S = \emptyset$, we obtain by construction, $k \geq \frac{1}{2}d_{G_0}(a)$ and the definition of $\mu$ that $d_{\vec{G_1}}^+(S)\geq |X''|\geq \frac{1}{2}\mu\geq 2$. We may hence suppose that there is some $x \in X'' \cap S$. If there is some $x' \in X''-S$, we obtain by $(6iii)$ and construction that
\begin{align*}
 d_{\vec{G_1}}^+(S)&\geq d_{\vec{W}}^+(S \cap W)+|\{ax'\}|\\&\geq \lambda_{\vec{W}}(x,x')+1\\
&\geq \min\{d_{\vec{W}}^+(x),d_{\vec{W}}^-(x')\}+1\\
&\geq2.
\end{align*}

Further, if there is some $z \in V(W)-(X''\cup Y''\cup S)$, we obtain by $(6iii)$ and construction that $d_{\vec{G_1}}^+(S)\geq  \lambda_{\vec{W}}(x,z)=\min\{d_{\vec{W}}^+(x),d_{\vec{W}}^-(z)\}\geq 2$.
 We may hence suppose that $v \in Y''$ and $V(W)-Y'' \subseteq S$. Let $v'$ be the unique neighbor of $v$ in $G_1$ which is not contained in $V(W)$. 
As $v \in Y''$ and by construction, we have that $A(\vec{G_1})$ contains the arc $v'v$. If $v' \in S$, then $d_{\vec{G_1}}^+(S)\geq d_{\vec{W}}^+(S\cap V(W))+|\{v'v\}|\geq 1+1=2$. We may hence suppose that $v' \in V(G_1)-S$. By Claim \ref{gross}, $(\beta)$, and $(\gamma)$, we obtain that $\lambda_{\vec{G_1}}(a,v')\geq 1$. Hence $\vec{G_1}$ contains a directed $av'$-path $P$. Let $z_0z_1$ be the last arc on $P$ that satisfies $z_0 \in S$ and $z_1 \in V(G_1)-S$. If $z_1\in V(W)$, we obtain that $V(P_{z_1,v'})\subseteq V(G_1)-S$ and $P_{z_1,v'}$ contains at least one arc $z_2z_3$ with $z_2 \in V(W)$ and $z_3 \in V(G_1)-(V(W)\cup a)$. By construction, we obtain that $z_2 \in Y'$ which contradicts $V(W)-Y''\subseteq S$. We hence obtain $ z_1 \in V(G_1)-(V(W)\cup a)$, so $d_{\vec{G_1}}^+(S)\geq d_{\vec{W}}^+(S \cap V(W))+|\{z_0z_1\}|\geq 1+1=2$.

Hence in all cases, we have $d_{\vec{G_1}}^+(S)\geq 2$. Due to the fact that $|X'|=\frac{1}{2}\mu\geq 2$, similar arguments show that $d_{\vec{G_1}}^-(S)\geq 2$.
\end{proof}
Now by Claims \ref{gross} and \ref{klein}, $(\alpha)$ and Proposition \ref{char}, we obtain that $\vec{G_1}$ is well-balanced.
\end{proof}

We now obtain $G_2$ from $G_1$ by adding a set $B$ containing two vertices $v',v''$ for every $v \in V_4$ and linking each of them to $a$ by $d_{G_1}(v)$ edges and adding a vertex $a'$ and linking it to to $a$ by 3 edges. Further, we define $\ell_2:V(G_2)\rightarrow \mathbb{Z}_{\geq{0}}$ by $\ell_2(v)=\ell_1(v)$ for all $v \in V(G_1)$ and $\ell_2(v)=d_{G_2}(v)$ for all $v \in V(G_2)-V(G_1)$. It follows immediately from Proposition \ref{ext} that $(G_2,\ell_2)$ is a positive instance of UBWBO if and only if $(G_1,\ell_1)$ is a positive instance of UBWBO. It hence follows from Lemma \ref{01} that $(G_2,\ell_2)$ is a positive instance of UBWBO if and only if $(G_0,\ell_0)$ is a positive instance of UBWBO. Further, as $(G_0,\ell_0)$ satisfies $(\alpha)-(\eta)$, it follows from Claim \ref{cong1} and construction that $(G_2,\ell_2)$ satisfies $(a)-(g)$ for the partition $(V_3,V_3',V_4\cup V(W)\cup B,\{a,a'\})$. Hence $(G_2,\ell_2)$ is an instance of SSUBWBO. As the size of $(G_2,\ell_2)$ is polynomial in the size of $(G_0,\ell_0)$ by construction and the choice of $W$, Lemma \ref{2hard} follows.
\subsection{Main proof}\label{secondoptimal} 
We are now ready to prove Theorem \ref{opthard}.
 Clearly, OCO is in NP. We prove the hardness by a reduction from SSUBWBO. We need the following definition: For some positive integers $n,\alpha$, a {\it $(n,\alpha)$-tube} is obtained from a path on $n$ vertices by replacing each edge by $\alpha$ copies of itself. Given a graph $G$ and some $v \in V(G)$, we denote the operation of adding a vertex-disjoint $(n,\alpha)$-tube to $G$ and identifying $v$ with one of the vertices of degree $\alpha$ of the $(n,\alpha)$-tube by {\it attaching} the $(n,\alpha)$-tube at $v$. Intuitively speaking, attaching a tube to a vertex augments the importance of the connectivity properties of that vertex.  Now let $(G,\ell)$ be an instance of SSUBWBO and let $(V_3,V_3',V_4,\{a,a'\})$ be a partition of $V(G)$ as described in the definition of SSUBWBO. Further, let $n=|V(G)|$. Clearly, we may suppose that $n \geq 13$. We now create a graph $H$ from $G$ in the following way: we attach a $(n^5,d_G(v))$-tube $T_v$ to every $v \in V_4 \cup \{a,a'\}$ and we attach a $(n^2,3)$-tube $T_v$ to every $v \in V_3$. We say that a subset $\{u,v\}$ of $V(G)$ of size 2 is {\it important} if $\{u,v\}\cap V_3'=\emptyset$ and $|\{u,v\}\cap V_3|\leq 1$ and an important set $\{u,v\}\subseteq V(G)$ is {\it super important} if $\{u,v\}\cap V_3=\emptyset$. We denote by $\mathcal{I}$ and $\mathcal{S}\mathcal{I}$ the collection of important and super important sets, respectively. 
 Further, we set $k={n^5 \choose 2}\sum_{v \in V_4 \cup \{a,a'\}}d_G(v)+n^{10}\sum_{\{u,v\}\in \mathcal{S}\mathcal{I}}\min\{d_G(u),d_G(v)\}+3n^7|\mathcal{I}-\mathcal{S}\mathcal{I}|$. For an illustration, see Figure \ref{firstsecond}. 

\begin{figure}[h]\begin{center}
  \includegraphics[width=\textwidth]{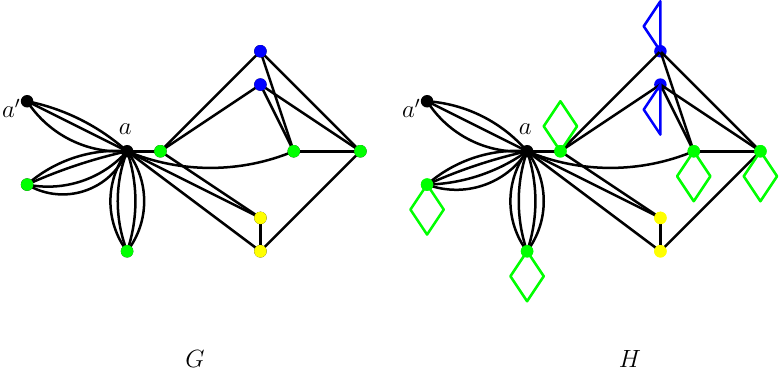} 
  \caption{An illustration of an instance $(G,\ell)$ of SSUBWBO and the corresponding instance $H$ of OCO. The vertices of $V_3,V_3'$ and $V_4$ are marked in blue, yellow, and green, respectively. Triangles and diamonds indicate, $(n^2,d_G(v))$-tubes and $(n^5,d_G(v))$-tubes, respectively.}\label{firstsecond}
\end{center}
\end{figure}

We show in the following that $(H,k)$ is a positive instance of OCO if and only if $(G,\ell)$ is a positive instance of SSUBWBO. First suppose that $(G,\ell)$ is a positive instance SSUBWBO, so there is a well-balanced orientation $\vec{G}$ of $G$ that satisfies $d_{\vec{G}}^+(v)\leq \ell(v)$ for all $v \in V(G)$. By Proposition \ref{ext}, we may suppose that $d_{\vec{G}}(a',a)=2$ and $d_{\vec{G}}(a,a')=1$. We first conclude an important property of $\vec{G}$.
\begin{Claim}\label{conn1}
For every pair $\{u,v\} \in \mathcal{I}$, we have $\lambda_{\vec{G}}(u,v)+\lambda_{\vec{G}}(v,u)=min\{d_G(u),d_G(v)\}$. 
\end{Claim}
\begin{proof}

First consider some $\{u,v\} \subseteq V_4 \cup a$. As $\vec{G}$ is well-balanced and by $(a)$ and $(d)$, we have $\lambda_{\vec{G}}(u,v)+\lambda_{\vec{G}}(v,u)=2\lfloor \frac{1}{2}\lambda_{G}(u,v)\rfloor=2\lfloor \frac{1}{2}\min\{d_G(u),d_G(v)\}\rfloor=min\{d_G(u),d_G(v)\}$. 

For the next case, we first show that $\lambda_{\vec{G}}(a,v)=2$ for every $v \in V_3$.
Let $S \subseteq V(G)$ be an $a\bar{v}$-set. Clearly, we may suppose that $V(G)-S$ is connected in $G$. 
 If $V(G)-S=\{v\}$, by $(c),(g)$ and $d_{\vec{G}}^+(v)\leq \ell(v)$, we have $d_{\vec{G}}^+(S)=d_G(v)-d_{\vec{G}}^+(v)\geq 3-1=2$. Otherwise, by $(e)$ and as $V(G)-S$ is connected in $G$, we obtain that $V(G)-S$ contains a vertex $v' \in V_4$. As $\vec{G}$ is well-balanced and by $(a)$ and $(d)$, we obtain $d_{\vec{G}}^+(S)\geq \lambda_{\vec{G}}(a,v')\geq \lfloor \frac{1}{2}\lambda_{G}(a,v')\rfloor=\lfloor\frac{1}{2}\min\{d_G(a),d_G(v')\}\rfloor=\lfloor\frac{1}{2} d_G(v')\rfloor\geq 2$. Hence $\lambda_{\vec{G}}(a,v)=2$ for every $v \in V_3$.

Now consider $u \in V_4 \cup a$ and $v \in V_3$. As $\vec{G}$ is well-balanced and by $(a),(c)$, and $(d)$, we have $\lambda_{\vec{G}}(v,u)\geq \lfloor \frac{1}{2}\lambda_{G}(u,v)\rfloor=\lfloor \frac{1}{2}\min\{d_G(u),d_G(v)\}\rfloor=\lfloor \frac{1}{2}d_G(v)\rfloor=1$. Further, by construction, as $\vec{G}$ is well-balanced, by $(a),(c)$ and $(d)$, we have $\lambda_{\vec{G}}(u,v)\geq \min\{\lambda_{\vec{G}}(u,a),\lambda_{\vec{G}}(a,v)\}=\min\{\lfloor \frac{1}{2}\lambda_{G}(u,a)\rfloor,2\}\geq \min\{\lfloor \frac{1}{2}d_G(u)\rfloor,2\}=2$. We obtain $\lambda_{\vec{G}}(u,v)+\lambda_{\vec{G}}(v,u)\geq 2+1=3=min\{d_G(u),d_G(v)\}$. As $\lambda_{\vec{G}}(u,v)+\lambda_{\vec{G}}(v,u)\leq min\{d_G(u),d_G(v)\}$ clearly holds, we obtain $\lambda_{\vec{G}}(u,v)+\lambda_{\vec{G}}(v,u)=min\{d_G(u),d_G(v)\}$.

Finally, for $v \in V_4 \cup V_3 \cup a$, by the assumption that $d_{\vec{G}}(a',a)=2$, we have $\lambda_{\vec{G}}(a',v)+\lambda_{\vec{G}}(v,a')=\min\{\lambda_{\vec{G}}(a',a),\lambda_{\vec{G}}(a,v)\}+\min\{\lambda_{\vec{G}}(v,a),\lambda_{\vec{G}}(a,a')\}=2+1=3$.
\end{proof}

We now create an orientation $\vec{H}$ of $H$ in the following way: First, we orient all the edges of $E(G)$ in the same way they are oriented in $\vec{G}$. Next consider some $v \in V_4\cup V_3 \cup \{a,a'\}$ and let $\{x_1,\ldots,x_q\}=V(T_v)$ such that $x_1=v$ and $x_ix_{i+1}\in E(T_v)$ for $i=1,\ldots,q-1$. For $i=1,\ldots,q-1$, we orient $d_{\vec{G}}^-(v)$ of the edges linking $x_i$ and $x_{i+1}$ from $x_i$ to $x_{i+1}$ and we orient the remaining $d_{\vec{G}}^+(v)$ edges  linking $x_{i}$ and $x_{i+1}$ from $x_{i+1}$ to $x_i$. 

We now show that $\vec{H}$ has the desired properties. First consider some $\{x,y\}\subseteq V(T_v)$ for some $v \in V_4 \cup V_3 \cup \{a,a'\} $ such that $x$ is contained in each $vy$-path in $H$. By construction, we have $\lambda_{\vec{H}}(x,y)+\lambda_{\vec{H}}(y,x)=d_{\vec{G}}^-(v)+d_{\vec{G}}^+(v)=d_G(v)$. 
\begin{Claim}\label{conn2}
Let $\{u,v\} \in \mathcal{I}$, $x \in V(T_u)$ and $y \in V(T_v)$. Then $\lambda_{\vec{H}}(x,y)+\lambda_{\vec{H}}(y,x)\geq \min\{d_G(u),d_G(v)\}$.
\end{Claim}
\begin{proof}
By construction we have $\lambda_{\vec{H}}(x,y)=\min\{\lambda_{\vec{H}}(x,u),\lambda_{\vec{H}}(u,v),\lambda_{\vec{H}}(v,y)\}=\min\{d_{\vec{G}}^+(u),\lambda_{\vec{G}}(u,v),d_{\vec{G}}^-(v)\}=\lambda_{\vec{G}}(u,v)$ and similarly $\lambda_{\vec{H}}(y,x)=\lambda_{\vec{G}}(v,u)$. By Claim \ref{conn1}, we have $\lambda_{\vec{H}}(x,y)+\lambda_{\vec{H}}(y,x)=\lambda_{\vec{G}}(u,v)+\lambda_{\vec{G}}(v,u)=min\{d_G(u),d_G(v)\}$.
\end{proof}

We are now ready to finish the proof. By Claim \ref{conn2} and construction, we have

\begin{align*}
\sum_{\{x,y\}\subseteq V(H)}(\lambda_{\vec{H}}(x,y)+\lambda_{\vec{H}}(y,x))&\geq \sum_{v \in V_4 \cup \{a,a'\}}\sum_{\{x,y\}\subseteq V(T_v)}(\lambda_{\vec{H}}(x,y)+\lambda_{\vec{H}}(y,x))\\&\hspace{3mm}+\sum_{\{u,v\} \in \mathcal{I}}\sum_{x \in V(T_u)}\sum_{y \in V(T_v)}(\lambda_{\vec{H}}(x,y)+\lambda_{\vec{H}}(y,x))\\
&=\sum_{v \in V_4 \cup \{a,a'\}}\sum_{\{x,y\}\subseteq V(T_v)}d_G(v)\\&\hspace{3mm}+\sum_{\{u,v\} \in \mathcal{I}}\sum_{x \in V(T_u)}\sum_{y \in V(T_v)}\min\{d_G(u),d_G(v)\}\\
&=\sum_{v \in V_4 \cup \{a,a'\}}{|V(T_v)|\choose 2}d_G(v)\\&\hspace{3mm} +\sum_{\{u,v\} \in \mathcal{I}}|V(T_u)||V(T_v)|\min\{d_G(u),d_G(v)\}\\
&={n^5 \choose 2}\sum_{v \in V_4 \cup \{a,a'\}}d_G(v)+n^{10}\sum_{\{u,v\} \in \mathcal{S}\mathcal{I}}\min\{d_G(u),d_G(v)\}\\&\hspace{3mm}+n^{7}\sum_{\{u,v\} \in \mathcal{I}-\mathcal{S}\mathcal{I}}\min\{d_G(u),d_G(v)\}\\
&={n^5 \choose 2}\sum_{v \in V_4 \cup \{a,a'\}}d_G(v)+n^{10}\sum_{\{u,v\} \in \mathcal{S}\mathcal{I}}\min\{d_G(u),d_G(v)\}\\&\hspace{3mm}+3 n^{7} |\mathcal{I}-\mathcal{S}\mathcal{I}|\\
&=k.
\end{align*}

We hence obtain that $(H,k)$ is a positive instance of OCO.
\medskip

Now suppose that $(H,k)$ is a positive instance of OCO, so there is an orientation $\vec{H}$ of $H$ with $\sum_{\{x,y\}\subseteq V(H)}(\lambda_{\vec{H}}(x,y)+\lambda_{\vec{H}}(y,x))\geq k$. 
 Let $\vec{G}=\vec{H}[V(G)]$ and observe that $\vec{G}$ is an orientation of $G$. We will prove through several claims that either $\vec{G}$ is a well-balanced orientation of $G$ that satisfies $d_{\vec{G}}^+(v)\leq \ell(v)$ for all $v \in V(G)$ or the orientation obtained from $\vec{G}$ by inversing all arcs has this property. The first claim is the most technical one.

\begin{Claim}\label{connasymm}
For all $\{u,v\} \in \mathcal{I}$, we have $\lambda_{\vec{G}}(u,v)+\lambda_{\vec{G}}(v,u)=\min\{d_G(u),d_G(v)\}$.
\end{Claim}
\begin{proof}
Suppose otherwise, so $\lambda_{\vec{G}}(u,v)+\lambda_{\vec{G}}(v,u)\leq \min\{d_G(u),d_G(v)\}-1$ for some $\{u,v\}  \in \mathcal{I}$. We now define $\mathcal{U}=\{\{x,y\}\subseteq V(H)|\{x,y\}\cap V_3'\neq \emptyset \text{ or }\{x,y\}\subseteq \bigcup_{v \in V_3}V(T_v)\}$.

Observe that for every $\{x,y\}\subseteq V(H)$, exactly one of the following holds:
\begin{itemize}
\item $\{x,y\} \subseteq V(T_v)$ for some $v \in V_4 \cup \{a,a'\}$,
\item $x \in V(T_u)$ and $y \in V(T_v)$ for some $\{u,v\}\in \mathcal{I}$,
\item $\{x,y\} \in \mathcal{U}$.
\end{itemize}

By the assumption, we obtain 

\begin{align*}
\sum_{v \in V_4 \cup \{a,a'\}}\sum_{\{x,y\}\subseteq V(T_v)}(\lambda_{\vec{H}}(x,y)+\lambda_{\vec{H}}(y,x))&+\sum_{\{u,v\} \in \mathcal{I}}\sum_{x \in V(T_u)}\sum_{y \in V(T_v)}(\lambda_{\vec{H}}(x,y)+\lambda_{\vec{H}}(y,x))\\
&\leq\sum_{v \in V_4 \cup \{a,a'\}}\sum_{\{x,y\}\subseteq V(T_v)}d_G(v)\\&\hspace{3mm}+\sum_{\{u,v\} \in \mathcal{I}}\sum_{x \in V(T_u)}\sum_{y \in V(T_v)}\lambda_{\vec{G}}(u,v)+\lambda_{\vec{G}}(v,u)\\
&={n^5 \choose 2}\sum_{v \in V_4 \cup \{a,a'\}}d_G(v)+n^{10}\sum_{\{u,v\} \in \mathcal{S}\mathcal{I}}\lambda_{\vec{G}}(u,v)+\lambda_{\vec{G}}(v,u)\\&\hspace{3mm}+n^{7}\sum_{\{u,v\} \in \mathcal{I}-\mathcal{S}\mathcal{I}}\lambda_{\vec{G}}(u,v)+\lambda_{\vec{G}}(v,u)\\
&\leq {n^5 \choose 2}\sum_{v \in V_4 \cup \{a,a'\}}d_G(v)+n^{10}\sum_{\{u,v\} \in \mathcal{S}\mathcal{I}}\min\{d_G(u),d_G(v)\}\\&\hspace{3mm}+n^{7}\sum_{\{u,v\} \in \mathcal{I}-\mathcal{S}\mathcal{I}}\min\{d_G(u),d_G(v)\}-n^7\\
&=k-n^7.
\end{align*}

By definition of $\mathcal{U}, (c)$ and construction, we have $\lambda_{\vec{H}}(x,y)+\lambda_{\vec{H}}(y,x)\leq 3$ for every $\{x,y\}\in \mathcal{U}$. By definition of $\mathcal{U}$, construction and $\alpha(1-\alpha)\leq \frac{1}{4}$ for every $0 \leq \alpha \leq 1$, this yields 
\begin{align*}
\sum_{\{x,y\}\in \mathcal{U}}(\lambda_{\vec{H}}(x,y)+\lambda_{\vec{H}}(x,y))&\leq 3|\mathcal{U}|\\
&\leq 3({n^2|V_3 \cup V_3'|\choose 2}+n^5 |V_3'||V(G)-V_3'|)\\
&\leq 3 n^6+\frac{3}{4}n^7\\
&< n^7,
\end{align*}
as $n \geq 13$.
\medskip

We obtain \begin{align*}
\sum_{\{x,y\}\subseteq V(H)}(\lambda_{\vec{H}}(x,y)+\lambda_{\vec{H}}(y,x))&= \sum_{v \in V_4 \cup \{a,a'\}}\sum_{\{x,y\}\subseteq V(T_v)}(\lambda_{\vec{H}}(x,y)
+\lambda_{\vec{H}}(y,x))\\&\hspace{3mm}+\sum_{\{u,v\} \in \mathcal{I}}\sum_{x \in V(T_u)}\sum_{y \in V(T_v)}(\lambda_{\vec{H}}(x,y)+\lambda_{\vec{H}}(y,x))\\&\hspace{3mm}+\sum_{\{x,y\}\in \mathcal{U}}(\lambda_{\vec{H}}(x,y)+\lambda_{\vec{H}}(y,x))\\&< k-n^7+n^7\\&=k,
\end{align*}
 a contradiction.
\end{proof}
The next claim shows that for a large portion of the pairs of vertices, the connectivity among them satisfies some symmetry condition.
\begin{Claim}\label{connsymm}
For all $\{u,v\} \subseteq V_4\cup a$, we have $\min\{\lambda_{\vec{G}}(u,v),\lambda_{\vec{G}}(v,u)\}\geq\frac{1}{2}\min\{d_G(u),d_G(v)\}$.
\end{Claim}
\begin{proof}
Suppose otherwise, so there are $u,v \in V_4 \cup a$ with $\lambda_{\vec{G}}(u,v)<\frac{1}{2}\min\{d_G(u),d_G(v)\}$.

 By symmetry, we may suppose that $d_G(v)\leq d_G(u)$, in particular $v \in V_4$ by $(a)$. We obtain by Claim \ref{connasymm} and $\{u,v\}\in \mathcal{I}$ that $\lambda_{\vec{G}}(v,u)=\min\{d_G(u),d_G(v)\}-\lambda_{\vec{G}}(u,v)>\frac{1}{2}\min\{d_G(u),d_G(v)\}=\frac{1}{2}d_G(v)$. This yields $d_{\vec{G}}^-(v)=d_G(v)-d_{\vec{G}}^+(v)\leq d_G(v)-\lambda_{\vec{G}}(v,u)<\frac{1}{2}d_G(v)$. By the last part of $(d)$, there are two vertices $v',v''\in V_4 \cup a$ such that $d_G(v)=d_G(v')=d_G(v'')$. By Claim \ref{connasymm}, we have $d_G(v)=\min\{d_G(v),d_G(v')\}=\lambda_{\vec{G}}(v,v')+\lambda_{\vec{G}}(v',v)\leq d_{\vec{G}}^-(v')+d_{\vec{G}}^-(v)$. This yields $d_{\vec{G}}^+(v')=d_G(v')-d_{\vec{G}}^-(v')=d_G(v)-d_{\vec{G}}^-(v')\leq d_{\vec{G}}^-(v)<\frac{1}{2}d_G(v)$. Similarly, we have $d_{\vec{G}}^+(v'')<\frac{1}{2}d_G(v)$. We obtain $\lambda_{\vec{G}}(v',v'')+\lambda_{\vec{G}}(v'',v')\leq d_{\vec{G}}^+(v')+d_{\vec{G}}^+(v'')<d_G(v)=\min\{d_G(v'),d_G(v'')\}$, a contradiction to Claim \ref{connasymm}.
\end{proof}
We now show that also pairs involving vertices of degree 3 satisfy some connectivity conditions.
\begin{Claim}\label{strong}
$\vec{G}$ is strongly connected.
\end{Claim}
\begin{proof}
By Claim \ref{connsymm}, we have that $V_4 \cup a$ is contained in a strongly connected component $C$ of $\vec{G}$. Suppose for the sake of a contradiction that there is a distinct strongly connected component $C'$ of $\vec{G}$. We have either $V(C')\subseteq V_3 \cup V_3'$ or $V(C')=\{a'\}$. First suppose that $V(C')\subseteq V_3 \cup V_3'$. By symmetry, we may suppose that $d_{\vec{G}}^+(V(C'))=0$. By $(e)$ and $(f)$, there is some $v \in V_4 \cap N_G(V(C'))$. By Claim \ref{connsymm}, we have $d_{\vec{G}}^-(v)\geq \lambda_{\vec{G}}(a,v)=\frac{1}{2}d_G(v)$, so $d_{\vec{G}}^+(v)=d_G(v)-d_{\vec{G}}^-(v)\leq \frac{1}{2}d_G(v)$. This yields $d_{\vec{G}}^+(V(C')\cup v)=d_{\vec{G}}^+(V(C'))+d_{\vec{G}}^+(v)-d_G(v,V(C'))\leq \frac{1}{2}d_G(v)-1$. We obtain $\lambda_{\vec{G}}(v,a)\leq d_{\vec{G}}^+(V(C')\cup v)\leq \frac{1}{2}d_G(v)-1=\frac{1}{2}\min\{d_G(v),d_G(a)\}-1$, a contradiction to Claim \ref{connsymm}. Hence $\vec{G}$ is strongly connected in $V(G)-a'$. In particular, by $(c)$, we have $1 \leq d_{\vec{G}}^+(v)\leq 2$ for every $v \in V_3$. Now suppose that $V(C')=\{a'\}$, say $d_{\vec{G}}^+(a')=0$. We obtain $\lambda_{\vec{G}}(a',v)+\lambda_{\vec{G}}(v,a')\leq d_{\vec{G}}^+(a')+d_{\vec{G}}^+(v)=0+2=2$ for some arbitrary $v \in V_3$. This contradicts Claim \ref{connsymm}.
\end{proof}

Possibly reversing all arcs of $\vec{G}$, we may suppose that  $d_{\vec{G}}(a',a)\geq \lceil\frac{1}{2}d_{G}(a',a)\rceil=2$. We are now ready to conclude that the degree condition is satisfied. 

\begin{Claim}\label{schluss}
For all $v \in V_3$, we have $d_{\vec{G}}^+(v)=1$.
\end{Claim}
\begin{proof}
Suppose otherwise, so there is some $v \in V_3$ with $d_{\vec{G}}^-(v)\leq d_G(v)-d_{\vec{G}}^+(v)\leq 3-2=1$. By the assumption on $\vec{G}$, we obtain $\lambda_{\vec{G}}(a',v)+\lambda_{\vec{G}}(v,a')\leq d_{\vec{G}}^-(v)+d_{\vec{G}}^-(a')\leq 1+1=2$, a contradiction to Claim \ref{connsymm} as $\{a',v\}\in \mathcal{I}$.
\end{proof}

By Claims \ref{connsymm} and \ref{strong}, and $(c)$, we have that $\vec{G}$ is well-balanced. By Claim \ref{schluss} and $(g)$, we have $d_{\vec{G}}^+(v)\leq \ell(v)$ for all $v \in V(G)$. Hence $(G,\ell)$ is a positive instance of SSUBWBO.

As $H$ is polynomial in the size of $G$ and by Lemma \ref{2hard}, we obtain Theorem \ref{opthard}.

\section{Approximation algorithm}\label{appro}


We here prove the following, more technical restatement of Theorem \ref{mainapp}.
\begin{Theorem}
There is a polynomial time algorithm whose input is an undirected graph $G$ and that computes an orientation $\vec{G}_1$ of $G$ such that $tac(\vec{G}_1)\geq \frac{2}{3}tac(\vec{G})$ for every orientation $\vec{G}$ of $G$.
\end{Theorem}
\begin{proof}
Let $G$ be a graph. Next, let $\vec{G_0}$ be an orientation that maximizes $\sum_{\{u,v\}\subseteq V(G)}(\min\{\lambda_{\vec{G_0}}(u,v),1\}+\min\{\lambda_{\vec{G_0}}(v,u),1\})$ over all orientations of $G$. Further, let $(S_1,\ldots,S_t)$ be the unique partition of $V(G)$ such that $H_i=G[S_i]$ is 2-edge-connected and $\lambda_G(u,v)\leq 1$ for every $\{u,v\}\subseteq V(G)$ such that $u \in S_i$ and $v \in S_j$ for some $i,j \in \{1,\ldots,t\}$ with $i \neq j$. Observe that this partition exists by Proposition \ref{2ec}. For $i=1,\ldots,t$, let $\vec{H_i}$ be a well-balanced orientation of $H_i$. We now create an orientation $\vec{G_1}$ of $G$ in the following way: For $i=1,\ldots,t$, we let all the edges in $E(H_i)$ have the orientation they have in $\vec{H_i}$. We let all the remaining edges have the orientation they have in $\vec{G_0}$. By Propositions \ref{2ec},\ref{gabow} and \ref{reach}, we can compute $\vec{G_1}$ in polynomial time. In the following, we prove that $\vec{G_1}$ has the desired properties. We let $\mathcal{P}=\{\{u,v\}\subseteq V(G)\}, \mathcal{P}_{\leq 1}=\{\{u,v\}\subseteq V(G)|\lambda_G(u,v)\leq 1\}$ and $\mathcal{P}_{\geq 2}=\{\{u,v\}\subseteq V(G)|\lambda_G(u,v)\geq 2\}$. Let $\vec{G}$ be an orientation of $G$. We will show that $tac(\vec{G}_1)\geq \frac{2}{3}tac(\vec{G})$.
\begin{Claim}\label{p1}
$\sum_{\{u,v\}\in \mathcal{P}_{\leq 1}}(\lambda_{\vec{G_1}}(u,v)+\lambda_{\vec{G_1}}(v,u))\geq \sum_{\{u,v\}\in \mathcal{P}_{\leq 1}}(\lambda_{\vec{G}}(u,v)+\lambda_{\vec{G}}(v,u)).$
\end{Claim}
\begin{proof}
First observe that, as $\vec{H_i}$ is strongly connected for $i=1,\ldots,t$, for every $\{u,v\} \in \mathcal{P}_{\leq 1}$, if $v$ is reachable from $u$ in $\vec{G_0}$, then $v$ is also reachable from $u$ in $\vec{G_1}$. Now let $\vec{G'}$ be the orientation of $G$ which is obtained by giving all the edges in $E(H_i)$ the orientation they have in $\vec{H_i}$ for $i=1,\ldots,t$ and giving all remaining edges the orientation they have in $\vec{G}$. As $\vec{H_i}$ is strongly connected for $i=1,\ldots,t$, we obtain that $\min\{\lambda_{\vec{G'}}(u,v),1\}+\min\{\lambda_{\vec{G'}}(v,u),1\}=2$ for every $\{u,v\}\in \mathcal{P}_2$ and for every $\{u,v\} \in \mathcal{P}_1$, if $v$ is reachable from $u$ in $\vec{G_0}$, then $v$ is also reachable from $u$ in $\vec{G_1}$. By the definition of $\vec{G_0}$, this yields
\begin{align*}
\sum_{\{u,v\}\in \mathcal{P}_{\leq 1}}(\lambda_{\vec{G_1}}(u,v)+\lambda_{\vec{G_1}}(v,u))&\geq \sum_{\{u,v\}\in \mathcal{P}_{\leq 1}}(\lambda_{\vec{G_0}}(u,v)+\lambda_{\vec{G_0}}(v,u))\\
&= \sum_{\{u,v\}\in \mathcal{P}_{\leq 1}}(\min\{\lambda_{\vec{G_0}}(u,v),1\}+\min\{\lambda_{\vec{G_0}}(v,u),1\})\\
&\geq \sum_{\{u,v\}\in \mathcal{P}}(\min\{\lambda_{\vec{G_0}}(u,v),1\}+\min\{\lambda_{\vec{G_0}}(v,u),1\})-2|\mathcal{P}_{\geq 2}|\\
&\geq \sum_{\{u,v\}\in \mathcal{P}}(\min\{\lambda_{\vec{G'}}(u,v),1\}+\min\{\lambda_{\vec{G'}}(v,u),1\})-2|\mathcal{P}_{\geq 2}|\\
&=\sum_{\{u,v\}\in \mathcal{P}_{\leq 1}}(\min\{\lambda_{\vec{G'}}(u,v),1\}+\min\{\lambda_{\vec{G'}}(v,u),1\})\\
&\geq \sum_{\{u,v\}\in \mathcal{P}_{\leq 1}}(\min\{\lambda_{\vec{G}}(u,v),1\}+\min\{\lambda_{\vec{G}}(v,u),1\})\\
&=\sum_{\{u,v\}\in \mathcal{P}_{\leq 1}}(\lambda_{\vec{G}}(u,v)+\lambda_{\vec{G}}(v,u)).
\end{align*}
\end{proof}
\begin{Claim}\label{p2}
$\sum_{\{u,v\}\in \mathcal{P}_{\geq 2}}(\lambda_{\vec{G_1}}(u,v)+\lambda_{\vec{G_1}}(v,u))\geq \frac{2}{3}\sum_{\{u,v\}\in \mathcal{P}_{\geq 2}}(\lambda_{\vec{G}}(u,v)+\lambda_{\vec{G}}(v,u)).$
\end{Claim}
\begin{proof}
Let $\{u,v\}\in \mathcal{P}_{\geq 2}$. By symmetry, we may suppose that $\{u,v\}\subseteq S_1$. As $\vec{H_1}$ is well-balanced, we obtain \begin{align*}\lambda_{\vec{G_1}}(u,v)+\lambda_{\vec{G_1}}(v,u)&\geq \lambda_{\vec{H_1}}(u,v)+\lambda_{\vec{H_1}}(v,u)\\&\geq 2\lfloor\frac{1}{2}\lambda_{H_1}(u,v)\rfloor\\&=2\lfloor\frac{1}{2}\lambda_{G}(u,v)\rfloor.\end{align*}

If $\lambda_G(u,v)=2$, we obtain $\lambda_{\vec{G_1}}(u,v)+\lambda_{\vec{G_1}}(v,u)\geq 2(\frac{1}{2}\lambda_{G}(u,v))=\lambda_{G}(u,v)\geq \lambda_{\vec{G}}(u,v)+\lambda_{\vec{G}}(v,u)$.

Otherwise, by $\lambda_G(u,v)\geq 2$, we have $\lambda_{\vec{G_1}}(u,v)+\lambda_{\vec{G_1}}(v,u)\geq \lambda_{G}(u,v)-1\geq \frac{2}{3}\lambda_G(u,v)\geq \frac{2}{3}(\lambda_{\vec{G}}(u,v)+\lambda_{\vec{G}}(v,u)).$

Hence the statement follows.
\end{proof}
By Claims \ref{p1} and \ref{p2}, we obtain 
\begin{align*}
tac(\vec{G_1})&=\sum_{\{u,v\}\in \mathcal{P}}(\lambda_{\vec{G_1}}(u,v)+\lambda_{\vec{G_1}}(v,u))\\
&=\sum_{\{u,v\}\in \mathcal{P}_{\leq 1}}(\lambda_{\vec{G_1}}(u,v)+\lambda_{\vec{G_1}}(v,u))+\sum_{\{u,v\}\in \mathcal{P}_{\geq 2}}(\lambda_{\vec{G_1}}(u,v)+\lambda_{\vec{G_1}}(v,u))\\
&\geq \sum_{\{u,v\}\in \mathcal{P}_{\leq 1}}(\lambda_{\vec{G}}(u,v)+\lambda_{\vec{G}}(v,u))+\frac{2}{3}\sum_{\{u,v\}\in \mathcal{P}_{\geq 2}}(\lambda_{\vec{G}}(u,v)+\lambda_{\vec{G}}(v,u))\\
&\geq \frac{2}{3}(\sum_{\{u,v\}\in \mathcal{P}_{\leq 1}}(\lambda_{\vec{G}}(u,v)+\lambda_{\vec{G}}(v,u))+\sum_{\{u,v\}\in \mathcal{P}_{\geq 2}}(\lambda_{\vec{G}}(u,v)+\lambda_{\vec{G}}(v,u))\\
&=\frac{2}{3}\sum_{\{u,v\}\in \mathcal{P}}(\lambda_{\vec{G}}(u,v)+\lambda_{\vec{G}}(v,u))\\&=\frac{2}{3}tac(\vec{G}).
\end{align*}
This finishes the proof.
\end{proof}
\section{Conclusion}\label{conc}
We have shown that it is NP-complete to find an orientation that maximizes $tac(\vec{G})$ over all orientations $\vec{G}$ of a given graph $G$, but that a $\frac{2}{3}$-approximation algorithm is available.
\medskip

A natural question is whether this approximation ratio can be improved. To this end, consider the graph $G$ which is obtained from a path $v_1\ldots v_t$ by tripling every edge. Further, let $\vec{G_1}$ be the orientation of $G$ in which two edges are oriented from $v_i$ to $v_{i+1}$ and one edge is oriented from $v_{i+1}$  to $v_i$ if $i$ is odd and in which one edge is oriented from $v_i$ to $v_{i+1}$ and two edges are oriented from $v_{i+1}$  to $v_i$ if $i$ is even for $i=1,\ldots,t-1$. Finally, let $\vec{G_2}$ be the orientation of $G$ in which all three edges are oriented from $v_i$ to $v_{i+1}$ for $i=1,\ldots,t-1$. For an illustration, see Figure \ref{path}.

\begin{figure}[h]\begin{center}
  \includegraphics[width=0.8\textwidth]{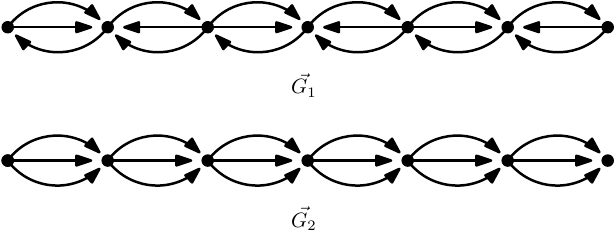} 
  \caption{An illustration of the construction for $t=7$.}\label{path}
\end{center}
\end{figure}

 It is easy to see that $G$ is 2-edge-connected and $\vec{G_1}$ is well-balanced. Hence $\vec{G_1}$ is a possible output of the algorithm in Section \ref{appro} when applied to $G$. Next observe that $tac(\vec{G_1})=2{t \choose 2}+(t-1)$ and $tac(\vec{G_2})=3{t \choose 2}$. As $\lim_{t \rightarrow \infty}\frac{2{t \choose 2}+(t-1)}{3{t \choose 2}}=\frac{2}{3}$, we obtain that the approximation guarantee of the algorithm in Section \ref{appro} cannot be improved. Hence an improvement on the approximation ratio would require a distinct algorithm. As a first step, it would be interesting to understand whether the maximization problem is APX-hard.
\medskip

Further, it would be interesting to understand how closely orientations maximizing total arc-connectivity and orientations maximizing classical connectivity notions are related. In \cite{cdgmo}, Casablanca et al. conjectured that every orientation $\vec{G}$ of a 2-edge-connected graph $G$ that maximizes $tvc(\vec{G})$ is strongly connected. The analogous statement is not true for total arc-connectivity as the digraph $\vec{G_2}$ in Figure \ref{path} shows. However, the following statement seems plausible.

\begin{Conjecture}
Every 2-edge-connected graph $G$ has an orientation $\vec{G}$ that maximizes $tac(\vec{G})$ and is strongly connected.
\end{Conjecture}

Even the following much stronger statement could be true.

\begin{Conjecture}
Every graph $G$ has an orientation $\vec{G}$ that maximizes $tac(\vec{G})$ and is well-balanced.
\end{Conjecture}
\medskip

Finally, one could wonder if an analogue of Theorem \ref{mainapp} exists for vertex-connectivity.
\begin{Question}
Is there a polynomial time algorithm that for a given graph $G$ approximates the maximum of $tvc(\vec{G})$ over all orientations $\vec{G}$ of $G$ within a constant factor? 
\end{Question}


\begin{thebibliography}{99}
\bibitem{bg} J. Bang-Jensen, G. Gutin, {\it Digraphs: Theory, Algorithms and Applications}, 2nd edition, Springer-Verlag, 2009,
\bibitem{AlexTibor} A. R. Berg, T. Jord\'{a}n, {\it Two-connected orientations of Eulerian graphs}, Journal of Graph Theory, 52(3): 230--242, 2006,
\bibitem{bikks} A. Bern\'ath, S. Iwata, T. Király, Z. Kir\'aly, Z. Szigeti, {\it Recent results on well-balanced orientations}, Discrete Optimization, 5:663-676, 2008,
\bibitem{bj} A. Bern\'ath, G. Joret, {\it Well-balanced orientations of mixed graphs}, Information Processing Letters, 106(4): 149–151, 2008,
\bibitem{cdgmo} R. Casablanca, P.  Dankelmann, W. Goddard, L.  Mol, O. Oellermann, {\it The maximum average connectivity among all orientations of a graph}, Journal of Combinatorial Optimization. 43(1), 2022, 
\bibitem{cds} J. Cheriyan, O. Durand de Gevigney, Z. Szigeti, {\it Packing of rigid spanning subgraphs and spanning trees}, Journal of Combinatorial Theory, Series B, 105 : 17–25, 2014,
\bibitem{duraj} L. Duraj, {\it Optimal graph orientation problems}, PhD thesis, \url{https://ruj.uj.edu.pl/xmlui/bitstream/handle/item/38768/duraj_optimal_graph_orientation_problems.pdf?sequence=1&isAllowed=y}, 2010,
\bibitem{ODG} O. Durand de Gevigney, {\it On Frank's conjecture on $k$-connected orientations}, Journal of Combinatorial Theory, Series B, 141, 105-114, 2020,
\bibitem{g} H. N. Gabow, {\it Efficient splitting off algorithms for graphs}, in: STOC ’94: Proceedings of the Twenty-Sixth AnnualACM Symposium on Theory of Computing, ACM Press: 696–705, 1994,
\bibitem{hak} S. L. Hakimi, {\it On the degrees of the vertices of a directed graph}, J. Franklin Inst 279(4): 290-308, 1965,
\bibitem{hsy}S. L. Hakimi, E. F. Schmeichel, N.l E. Young, {\it Orienting Graphs to Optimize Reachability}, Information Processing Letters 63(5), 229-235, 1997,
\bibitem{ho} M. A. Henning, O. R. Oellermann, {\it The average connectivity of a digraph}, Discrete Applied Mathematics, 140 (1–3), 143-153, 2004,
\bibitem{hs} F. Hörsch, Z. Szigeti, {\it On the complexity of finding well-balanced orientations with upper bounds on the out-degrees}, Journal of Combinatorial Optimization 45(30), 2023,
\bibitem{ks} Z. Kir\'aly, Z. Szigeti, {\it Simultaneous well-balanced orientations of graphs}, Journal of Combinatorial Theory, Series B, 96(5):684-692, 2006,
\bibitem{N60} C. St. J. A. Nash--Williams, {\it On orientations, connectivity, and odd vertex pairings in finite graphs}, Canadian Journal of Mathematics, 12:555--567, 1960,
\bibitem {robb} H. E. Robbins, {\it A theorem on graphs with an application to a problem of traffic control}, American Math. Monthly 46, 281-283, 1939,
\bibitem{CT} C. Thomassen, {\it Strongly 2-connected orientations of graphs}, Journal of Combinatorial Theory, Series B, 110:67--78, 2015.
\end{thebibliography}
\end{document}